\documentclass[a4paper,10pt]{article}
\usepackage[latin1]{inputenc}
\usepackage[T1]{fontenc}
\usepackage{amsmath}
\usepackage{amsfonts}
\usepackage{amstext}
\usepackage{amsthm}
\usepackage[all]{xy}
\usepackage{geometry}
\usepackage{srcltx}
\usepackage{graphics}
\usepackage{epsfig}
\usepackage{subfigure}
\usepackage{slashed}
\usepackage{color}
\usepackage{bbm}
\usepackage{amssymb}
\usepackage{srcltx}
\newtheorem{theorem}{Theorem}[section]
\newtheorem{lemma}[theorem]{Lemma}
\newtheorem{defi}[theorem]{Definition}
\newtheorem{rem}[theorem]{Remark}
\newtheorem{prop}[theorem]{Proposition}
\newtheorem{cor}[theorem]{Corollary}

\DeclareMathOperator{\im}{im}

\DeclareMathOperator{\ind}{ind}

\DeclareMathOperator{\coker}{coker}

\DeclareMathOperator{\re}{Re}

\DeclareMathOperator{\spann}{span}

\newcommand{\N}{{\mathbb N}}
\newcommand{\R}{{\mathbb R}}
\newcommand{\Z}{{\mathbb Z}}
\title{The Index Bundle and Multiparameter Bifurcation for Discrete Dynamical Systems}

\author{Robert Skiba and Nils Waterstraat}

\begin{document}
\date{}
\maketitle


\begin{abstract}
We develop a $K$-theoretic approach to multiparameter bifurcation theory of homoclinic solutions of discrete non-autonomous dynamical systems from a branch of stationary solutions. As a byproduct we obtain a family index theorem for asymptotically hyperbolic linear dynamical systems which is of independent interest. In the special case of a single parameter, our bifurcation theorem weakens the assumptions in previous work by Pejsachowicz and the first author.
\end{abstract}

\section{Introduction}
The aim of this paper is to investigate the existence of non-trivial homoclinic trajectories
of discrete non-autonomous dynamical systems by topological bifurcation theory. For many years, bifurcation theory for various types of bounded solutions of one-parameter families of discrete non-autonomous
dynamical systems has been studied by many authors using exponential dichotomies
for linear operator equations, the Lyapunov-Schmidt method and the
Melnikov integral. We refer the reader to \cite{Huls,Pa84,Pal88,Christian10,Poetzsche, Poetzscheb}, for example,  and
the extensive bibliography that can be found therein. Recently, following \cite{JacoboAMS}, the first author proposed jointly with Pejsachowicz in \cite{JacoboRobertI} and \cite{JacoboRobertII} a new approach to this subject by topological methods using the index bundle for closed paths of Fredholm operators and the $KO$-theory of the unit circle. This paper provides a significant extension of the results from \cite{JacoboRobertII} from systems depending on a single parameter to the multiparameter case, where we also weaken the assumptions in the previous work. \newline\indent
To summarise the setting briefly, we fix some $N\in\mathbb{N}$ and denote by $c_0(\mathbb{R}^N)$ the Banach space of all sequences in $\mathbb{R}^N$ converging to $0$ as $n\rightarrow\pm\infty$ with respect to the sup norm. Let $\Lambda$ be a compact CW-complex and $f=\{f_n\colon\Lambda\times\R^N\rightarrow\R^N\}_{n\in\mathbb{Z}}$ a sequence of continuous maps such that each $f_n$ is differentiable with respect to the $\mathbb{R}^N$ variable and the derivative $(Df_n)(\lambda,u)$ depends continuously on $(\lambda,u)\in\Lambda\times\R^N$. We assume that $f_n(\lambda,0)=0$ for all $n\in\mathbb{Z}$ and we consider the family of discrete dynamical systems
\begin{align}\label{nonlin}
x_{n+1}=f_n(\lambda,x_n),\quad n\in\mathbb{Z},
\end{align}
parametrised by $\Lambda$. We refer to solutions of \eqref{nonlin} which belong to $c_0(\mathbb{R}^N)$ as \textit{homoclinic solutions}. Note that by assumption $0=\{0\}_{n\in\mathbb{N}}\in c_0(\R^N)$ satisfies \eqref{nonlin} for all $\lambda\in\Lambda$ and in this article we study \textit{bifurcation of homoclinic solutions of \eqref{nonlin}} from this given branch of solutions.\newline\noindent
A central role is played by the discrete dynamical systems
\begin{align}\label{lin}
x_{n+1}=a_n(\lambda)x_n,\quad n\in\mathbb{Z},
\end{align}
where $a_n(\lambda)=(Df_n)(\lambda,0)$ are the derivatives of $f_n(\lambda,\cdot):\mathbb{R}^N\rightarrow\mathbb{R}^N$ at $0\in\mathbb{R}^N$. Under suitable assumptions on the sequence of maps $f=\{f_n\}_{n\in\mathbb{Z}}$ in \eqref{nonlin} and $a=\{a_n\}_{n\in\mathbb{Z}}$ in \eqref{lin}, there is an open ball $B\subset c_0(\R^N)$ about $0$, as well as maps $F\colon\Lambda\times B\rightarrow c_0(\R^N)$ and $L\colon\Lambda\times c_0(\R^N)\rightarrow c_0(\R^N)$ such that $F(\lambda,x)=0$ if and only if $x\in B$ is a homoclinic solution of \eqref{nonlin}, and $L_\lambda x=0$ if and only if $x$ is a homoclinic solution of \eqref{lin}. Moreover, each $F_\lambda$ is a $C^1$-Fredholm map and $D_0F_\lambda$, the derivative of $F_\lambda:B\rightarrow c_0(\R^N)$ at $0\in B$, is given by $L_\lambda$. In this way, we study the bifurcation problem for \eqref{nonlin} by applying topological methods to the family of nonlinear maps $F$ and its linearisation $L$.\newline\noindent
Let us now briefly sketch our methods and our main results. We will see below that the operators $L_\lambda$ are Fredholm for all $\lambda\in\Lambda$. Using a construction due to Atiyah and Jänich from the sixties, we obtain an element $\ind(L)$ in the $KO$-theory group of the parameter space $\Lambda$ which is a natural generalisation of the classical integral Fredholm index of a single operator. We will recall below the definition of $KO(\Lambda)$, which is a group made by equivalence classes of vector bundles over $\Lambda$. This index bundle was used by Atiyah and Singer in the extension of their famous index theorem to families in \cite{AtiyahSinger}, and among many other applications, its relevance for multiparameter bifurcation theory was discovered by Pejsachowicz in \cite{JacoboK} and since then used in various generality in, e.g., \cite{BartschII}, \cite{FiPejsachowiczII}, \cite{JacoboTMNA0}, \cite{JacoboTMNAI}, \cite{JacoboTMNAII}, \cite{JacoboJFPTA} and \cite{NilsBif}. Here, roughly speaking, after a rather technical finite dimensional reduction of the bifurcation problem, it turns out that the non-existence of a bifurcation point for the nonlinear map $F$ would imply that $J(\ind(L))$ vanishes, where $J:KO(\Lambda)\rightarrow J(\Lambda)$ denotes the $J$-homomorphism, that was introduced by Atiyah in \cite{AtiyahThom}. The $J$-homomorphism is notoriously hard to compute, however, its non-triviality can be obtained from characteristic classes of vector bundles what makes it substantially easier to work with it in our situation. Here we obtain a bifurcation theorem from this idea by computing explicitly the index bundle $\ind(L)$ in $KO(\Lambda)$ in terms of vector bundles obtained from eigenvectors of the matrices $a_n(\lambda)$ in \eqref{lin}. This family index theorem for the linear discrete dynamical systems \eqref{lin} is of independent interest and it is the first main result of this paper that we prove below. In combination with the above outlined method to study bifurcation of solutions for the family of operators $F$, we then obtain our second main theorem which gives a topological bifurcation invariant for the nonlinear discrete dynamical systems \eqref{nonlin}. Let us finally point out that if $\Lambda=S^1$, i.e. our families are paths, then our invariant becomes $\mathbb{Z}_2$-valued and coincides with the previously introduced one by the first author and Pejsachowicz in \cite{JacoboRobertII}. However, as we have already mentioned above, the assumptions in our bifurcation theorem are weaker than in the previous approach.\\
The paper is organised as follows. In the next section we introduce our assumptions and state our main theorems as well as some corollaries.
In Section $3$ we first recall the concept of the Atiyah-Jänich index bundle and its main properties. Afterwards, we prove our family index theorem \ref{thm-lin} for asymptotically hyperbolic linear discrete dynamical systems. In the fourth section, we prove our main Theorem \ref{thm:nonlin} on bifurcation of homoclinic solutions for the nonlinear systems \eqref{nonlin}.
Section $5$ contains the detailed proof of Theorem \ref{thm:nonlinII}, which derives from Theorem \ref{thm:nonlin} an estimate of the covering dimension of all bifurcation points in $\Lambda$. The following sixth section provides a nontrivial example illustrating our results about the existence of bifurcation points. Afterwards, in Section $7$ we discuss an example of a discrete dynamical system which explains the role of the assumptions in our theorems. The final Section $8$ is devoted to some comments about possible extensions of our results.\\
Let us finally fix some common notations that we use throughout the paper. We denote by $\mathcal{L}(X)$ the space of bounded linear operators on a Banach space $X$ with the operator norm, and by $GL(X)$ the open subset of all invertible elements. The symbol $I_X$ stands for the identity operator on $X$. If $X=\mathbb{K}^N$ for $\mathbb{K}=\mathbb{C}$ or $\mathbb{K}=\mathbb{R}$ and for some $N\in\mathbb{N}$, then we use instead the common notation $M(N,\mathbb{K})$, $GL(N,\mathbb{K})$ and $I_N$, respectively. In what follows, we are going to work with vector bundles $E$ over topological spaces $\Lambda$, and we shall denote the product bundle $E=\Lambda\times V$ for a linear space $V$ by $\Theta(V)$.

%%%%%%%%%%%%%%%%%%%%%%%%%%%%%%%%%%%%%%%%%%%%%%%%%%%%%%%%%%%%%%%%%%%%%%%%%%%%%%%%%%%%%%%%%%%%%%%%%%%%%%%%%%%%%%%%%%%%%%%%%%%%%%%%%%%%%%%%%%%%%%%%%%%%%%%%%%%%%%%%%%%%%%%%%%%%%%%%%%%%%%%%%%%%%%%%%%%%%%%%%%%%%%%%%%%%%%%%%%%%%%%%%%%%%%%%%%%%%%%%%%%%%%%%%%%%%%%%%%%%%%%%%%%%%%%%%%%%%%%%%%%%%%%%%%%%%%%%%%%%%%%%%%%%%%%%%%%%%%%%%%%%%%%%%%%%%%%%%%%%%%%%%%%%%%%%%%%%%%%%%%%%%%%%%%%%%%%%%%%%%%%%%%%%%%%%%%%%%%%%%%%%%%%%%%%%%%%%%%%%%%%%%%%%%%%%%%%%%%%%%%%%%%%%%%%%%%%%%%%%%%%%%%%%%%%%%%%%%%%%%%%%%%%%%%%%%%%%%%%%%%%%%%%%%%%%%%%%%%%%%%%%%%%%%%%%%%%%%%%%%%%%%%%%%%%%%%%%%%%%%%%%%%%%%%%%%%%%%%%%%%%%%%%%%%%%%%%%%%%%%%%%

\section{The Main Theorem and some Corollaries}\label{section-mainthm}
As in the introduction, we let $\Lambda$ be a connected compact CW complex, and we now fix a metric $d$ on $\Lambda$. We consider a sequence of continuous maps $f_n\colon\Lambda\times\mathbb{R}^N\rightarrow\mathbb{R}^N$, $n\in\mathbb{Z}$, which are differentiable with respect to the $\mathbb{R}^N$ variable and satisfy the following assumptions:
\begin{itemize}
	\item[(A1)] $f_n(\lambda,0)=0$ for all $n\in\mathbb{Z}$ and $\lambda\in\Lambda$,
	\item[(A2)] for every compact set $K\subset \R^N$ and for every $\varepsilon>0$ there exists $\delta>0$ such that
	\[\sup_{n\in\mathbb{Z}}\left\|f_n(\lambda,x)-f_n(\mu,y)\right\|+\sup_{n\in\mathbb{Z}}\left\| (Df_n)(\lambda,x)-(Df_n)(\mu,y)\right\|<\varepsilon\]
	for all $(\lambda,x),(\mu,y)\in\Lambda\times K$ such that $d(\lambda,\mu)+\|x-y\|<\delta$.
\end{itemize}
Note that $0\in c_0(\mathbb{R}^N)$ is a solution of all equations \eqref{nonlin} by (A1).

\begin{defi}\label{def:bif}
A \textit{bifurcation point} for the family of nonlinear difference equations \eqref{nonlin} is a parameter value $\lambda^\ast\in\Lambda$ such that in every neighbourhood of $(\lambda^\ast,0)$ in $\Lambda\times c_0(\mathbb{R}^N)$ there is $(\lambda,x)$ such that $x\neq 0$ satisfies \eqref{nonlin}.
\end{defi}
\noindent

In what follows, we set
\begin{equation}\label{eq-a-n}
a_n(\lambda):=(Df_n)(\lambda,0),\quad n\in\mathbb{Z},\,\lambda\in\Lambda,
\end{equation}
and we note that this is a sequence of continuous families of matrices $a_n\colon \Lambda\rightarrow M(N,\mathbb{R})$. Let us recall that an invertible matrix is called \textit{hyperbolic} if it has no eigenvalue of modulus $1$. Further to the assumptions $($A1$)$--$($A2$)$ above, we also require
\begin{itemize}
	\item[(A3)] the sequence $a_n(\lambda)$ converges uniformly in $\lambda$ to families $a(\lambda,\pm\infty)$ of hyperbolic matrices,
	\item[(A4)] the matrices $a(\lambda,+\infty)$ and $a(\lambda,-\infty)$ have the same number of eigenvalues (counting with multiplicities) inside the unit disc for some, and hence for any, $\lambda\in\Lambda$,
	\item[(A5)] there is some $\lambda_0\in\Lambda$ such that the linear difference equation
	\begin{align}\label{equlin}
	x_{n+1}=a_n(\lambda_0)x_n,\quad n\in\mathbb{N},
	\end{align}
	has only the trivial solution $0\in c_0(\mathbb{R}^N)$.
\end{itemize}
Let us recall that a continuous map $a\colon \Z\times\Lambda\rightarrow M(N,\mathbb{R})$ which satisfies Assumption (A3) is called \textit{asymptotically hyperbolic}.\\

%Before proceeding with our main theorem, let us first make some remarks about our assumptions.
%{\color{red}
%\begin{rem}
%\emph{A continuous map $a\colon \Z\times\Lambda\rightarrow\mathcal{L}(\mathbb{R}^N)$ which satisfies Assumption (A3) is called \textit{asymptotically hyperbolic}. }
%\end{rem}}
%\noindent

As a hyperbolic matrix $a\in GL(N,\mathbb{R})$ has no eigenvalues on the unit circle, the spectral projection
\begin{align}\label{comP}
P=\frac{1}{2\pi i}\int_{S^1}{(zI_N-a)^{-1}dz}\in M(N,\mathbb{C})
\end{align}
is defined. We denote by $\re\colon \mathbb{C}^N\rightarrow\mathbb{R}^N$ the real part of elements in $\mathbb{C}^N$ and set
\begin{align}\label{reP}
P^su=\re(Pu),\quad P^uu=u-P^su,\quad u\in\mathbb{R}^N,
\end{align}
which are projections in $\mathbb{R}^N$ as the matrix $a$ is real. The image of $P^s$ consists of the real parts of all generalised eigenvectors with respect to eigenvalues inside $S^1$ and it is called the \textit{stable subspace} $E^s(a)$ of $a$. Analogously, the image of $P^u$ consists of the real parts of all generalised eigenvectors having eigenvalues outside $S^1$ and it is called the unstable subspace $E^u(a)$ of $a$. Note that

\begin{align}\label{sum}
E^s(a)\oplus E^u(a)=\mathbb{R}^N
\end{align}
and, moreover, it is not very difficult to see that

\begin{align*}
E^s(a)&=\{x_0\in\mathbb{R}^N\mid \, a^nx_0\rightarrow 0,\, n\rightarrow\infty\},\\
E^u(a)&=\{x_0\in\mathbb{R}^N\mid \, a^{n}x_0\rightarrow 0,\, n\rightarrow-\infty\}.
\end{align*}
Let us now consider the families of hyperbolic matrices $a(\lambda,\pm\infty)$, $\lambda\in\Lambda$, that we introduced in (A4). As none of these matrices has an eigenvalue on $S^1$, \eqref{comP} defines two continuous families of projections on $\mathbb{C}^N$. By taking real parts and building complementary projections as in \eqref{reP}, we obtain four families of projections on $\mathbb{R}^N$ parametrised by the space $\Lambda$. In what follows we denote the images of these projections by $E^s(\lambda,\pm\infty)$ and $E^u(\lambda,\pm\infty)$. Finally, as the images of continuous families of projections are vector bundles over the parameter space (cf. e.g. \cite{Park}), we get four vector bundles over $\Lambda$
\begin{align*}
E^s(\pm\infty)&=\{(\lambda,u)\in\Lambda\times\mathbb{R}^N \mid\, u\in E^s(\lambda,\pm\infty)\},\\
E^u(\pm\infty)&=\{(\lambda,u)\in\Lambda\times\mathbb{R}^N\mid \, u\in E^u(\lambda,\pm\infty)\},
\end{align*}
which are all subbundles of the product bundle $\Theta(\mathbb{R}^N)=\Lambda\times\mathbb{R}^N$, and which satisfy
\begin{align}\label{bundlesum}
E^s(+\infty)\oplus E^u(+\infty)=E^s(-\infty)\oplus E^u(-\infty)=\Theta(\mathbb{R}^N)
\end{align}
by \eqref{sum}. Let us recall that two vector bundles $E$, $F$ over $\Lambda$ are called \textit{stably isomorphic} if $E\oplus\Theta(\mathbb{R}^n)$ is isomorphic to $F\oplus\Theta(\mathbb{R}^n)$ for some non-negative integer $n$. The $KO$-theory of $\Lambda$ is the abelian group $KO(\Lambda)$ consisting of all formal differences $[E]-[F]$ of isomorphism classes $[E]$, $[F]$ of vector bundles $E$ and $F$ over $\Lambda$, with the equivalence relation $[E]-[F]\simeq[E']-[F']$ if $E\oplus F'$ is stably isomorphic to $E'\oplus F$. The group product on $KO(\Lambda)$ is induced by the direct sum of vector bundles, i.e.
\begin{equation*}
([E]-[F])+([E']-[F'])=[E\oplus E']-[F\oplus F'].
\end{equation*}
Every continuous map $g\colon\Lambda\rightarrow\Lambda'$ between compact topological spaces induces a group homomorphism $g^\ast\colon KO(\Lambda')\rightarrow KO(\Lambda)$ by the pullback construction for vector bundles, which makes $KO$-theory a contravariant functor from the category of topological spaces to the category of abelian groups. For $\lambda_0\in\Lambda$, there is a canonical inclusion map $\iota\colon \{\lambda_0\}\hookrightarrow\Lambda$. The reduced $KO$-theory group of $\Lambda$, denoted by $\widetilde{KO}(\Lambda)$, is defined as the kernel of the induced homomorphism $\iota^\ast\colon KO(\Lambda)\rightarrow KO(\{\lambda_0\})$. As $\Lambda$ is connected, this definition does not depend on the choice of $\lambda_0$, and moreover it is readily seen that $\widetilde{KO}(\Lambda)$ consists of all $[E]-[F]\in KO(\Lambda)$ such that $E$ and $F$ are of the same dimension.\\
The vector bundles $E$ and $F$ over $\Lambda$ are called \textit{fibrewise homotopy equivalent} if there is a fibre preserving homotopy equivalence between their sphere bundles $S(E)$ and $S(F)$. Moreover, $E$ and $F$ are \textit{stably fibrewise homotopy equivalent} if $E\oplus\Theta(\mathbb{R}^m)$ and $F\oplus\Theta(\mathbb{R}^n)$ are fibrewise homotopy equivalent for some non-negative integers $m,n$. The quotient of $\widetilde{KO}(\Lambda)$ by the subgroup generated by all $[E]-[F]$ where $E$ and $F$ are stably fibrewise homotopy equivalent is denoted by $J(\Lambda)$ and the quotient map $J\colon\widetilde{KO}(\Lambda)\rightarrow J(\Lambda)$ is called the \textit{generalised J-homomorphism} (cf. \cite{KTheoryAtiyah}).\\
With all this said, we can now state the main theorem of this paper.

\begin{theorem}\label{thm:nonlin}
If the system \eqref{nonlin} satisfies the assumptions \emph{(A1)--(A5)} and
\[J(E^s(+\infty))\neq J(E^s(-\infty))\in J(\Lambda),\]
then there is a bifurcation point.
\end{theorem}

\begin{rem}
We will see in Remark \ref{rem:dim} below that $[E^s(+\infty)]-[E^s(-\infty)]\in\widetilde{KO}(\Lambda)$, so that the $J$-homomorphism can indeed be applied to this class.

\end{rem}

\begin{rem}
It follows from \eqref{bundlesum} that

\begin{align*}
[E^s(+\infty)]-[E^s(-\infty)]=[E^u(+\infty)]-[E^u(-\infty)]\in\widetilde{KO}(\Lambda)
\end{align*}
and so we could actually replace the stable bundles by the unstable ones in Theorem \ref{thm:nonlin}.
\end{rem}

Let us point out that very little is known about $J(\Lambda)$ as these groups are notoriously hard to compute, and so one might guess that Theorem \ref{thm:nonlin} is of limited use. However, in order to check that $J(E^s(+\infty))\neq J(E^s(-\infty))$ we not even need to know $J(\Lambda)$ explicitly. Indeed, the i-th Stiefel Whitney class $w_i(E)\in H^i(\Lambda;\mathbb{Z}_2)$, $i\in\mathbb{N}$, for vector bundles $E$ over $\Lambda$ descends to a map on $\widetilde{KO}(\Lambda)$. Moreover, $w_i(E)$ only depends on the stable fibrewise homotopy classes of the associated sphere bundle $S(E)$ and so $w_i$ factorises through $J(\Lambda)$. Consequently, if $w_i(E)\neq w_i(F)$, then $J(E)\neq J(F)$ for any bundles $E$, $F$ over $\Lambda$. Denoting by
\[w(E)=1+w_1(E)+w_2(E)+\ldots\in H^\ast(\Lambda;\mathbb{Z}_2)\]
the total Stiefel-Whitney class of $E$, we obtain the following corollary of Theorem \ref{thm:nonlin}.
\begin{cor}
If the system \eqref{nonlin} satisfies the assumptions \emph{(A1)--(A5)} and
\[w(E^s(+\infty))\neq w(E^s(-\infty))\in H^\ast(\Lambda;\mathbb{Z}_2),\]
then there is a bifurcation point.
\end{cor}
Finally, we want to note the special case when the parameter space is the unit circle, as this is the setting of \cite{JacoboRobertI}. Then $w_1(E^s(\pm\infty))$ is an element of $H^1(S^1;\mathbb{Z}_2)$ which is isomorphic to $\mathbb{Z}_2$. It is readily seen that in this case our $w_1(E^s(\pm\infty))$ can be identified with the $\mathbb{Z}_2$-valued invariant in \cite{JacoboRobertI} and so we obtain an alternative proof of the main theorem of \cite{JacoboRobertI}. Note, however, that our assumptions in Theorem \ref{thm:nonlin} are weaker as in \cite{JacoboRobertI}.

\begin{cor}\label{cor-Jacobo}
If $\Lambda=S^1$, the family \eqref{nonlin} satisfies the assumptions \emph{(A1)--(A5)} and
\[w_1(E^s(+\infty))\neq w_1(E^s(-\infty))\in H^1(S^1;\mathbb{Z}_2)=\mathbb{Z}_2,\]
then there is a bifurcation point.
\end{cor}
Fitzpatrick and Pejsachowicz introduced in \cite{FiPejsachowiczII} an argument to estimate the dimension of the set of all bifurcation points in several parameter bifurcation theory. Since then their method has been used plenty of times, e.g. in \cite{JacoboTMNA0}, \cite{JacoboTMNAII}, \cite{MaciejIch}, \cite{JacoboJFPTA}, and recently it has been revisited by the second author in \cite{NilsBif}. Here we use it to derive from Theorem \ref{thm:nonlin} the following result, where $B\subset\Lambda$ denotes the set of all bifurcation points of the family \eqref{nonlin}.
\begin{theorem}\label{thm:nonlinII}
If $\Lambda$ is a compact connected topological manifold of dimension $k\geq 2$, the family \eqref{nonlin} satisfies the assumptions \emph{(A1)--(A5)} and
\[w_i(E^s(+\infty))\neq w_i(E^s(-\infty))\in H^i(\Lambda;\mathbb{Z}_2)\]
for some $1\leq i\leq k-1$, then the covering dimension of $B$ is at least $k-i$ and $B$ is not contractible as a topological space.
\end{theorem}

%%%%%%%%%%%%%%%%%%%%%%%%%%%%%%%%%%%%%%%%%%%%%%%%%%%%%%%%%%%%%%%%%%%%%%%%%%%%%%%%%%%%%%%%%%%%%%%%%%%%%%%%%%%%%%%%%%%%%%%%%%%%%%%%%%%%%%%%%%%%%%%%%%%%%%%%%%%%%%%%%%%%%%%%%%%%%%%%%%%%%%%%%%%%%%%%%%%%%%%%%%%%%%%%%%%%%%%%%%%%%%%%%%%%%%%%%%%%%%%%%%%%%%%%%%%%%%%%%%%%%%%%%%%%%%%%%%%%%%%%%%%%%%%%%%%%%%%%%%%%%%%%%%%%%%%%%%%%%%%%%%%%%%%%%%%%%%%%%%%%%%%%%%%%%%%%%%%%%%%%%%

\section{The Index Bundle for Discrete Dynamical Systems}
The aim of this section is to compute the index bundle of families of linear discrete dynamical systems as \eqref{lin}. We will firstly recap the definition of the index bundle and discuss its properties on the Banach space $c_0(\mathbb{R}^N)$. Secondly, we prove an explicit family index theorem for \eqref{lin} which is of independent interest.

\subsection{The Family Index Theorem}\label{section:famthm}
Let us recall that a bounded operator $T\colon X\rightarrow Y$ acting between Banach spaces $X$, $Y$ is called \textit{Fredholm} if it has finite dimensional kernel and cokernel. The \textit{Fredholm index} of $T$ is the integer
\begin{align}\label{Fredind}
\ind(T)=\dim\ker(T)-\dim\coker(T).
\end{align}
We denote by $\Phi(X,Y)$ the subspace of $\mathcal{L}(X,Y)$ consisting of all Fredholm operators, and by $\Phi_k(X,Y)$, $k\in\mathbb{Z}$, the subset of all operators in $\Phi(X,Y)$ having index $k$. Let us recall that the sets $\Phi_k(X,Y)$ are the path components of $\Phi(X,Y)$.\\
Atiyah and Jänich introduced independently the \textit{index bundle} for families $L\colon \Lambda\rightarrow\Phi(X,Y)$ of Fredholm operators parametrised by a compact topological space $\Lambda$ (cf. e.g. \cite{KTheoryAtiyah}, \cite{indbundleIch}), which is a generalisation of the integral Fredholm index \eqref{Fredind} to families. The construction can be outlined as follows: By the compactness of $\Lambda$, there is a finite dimensional subspace $V\subset Y$ such that
\begin{align}\label{subspace}
\im(L_\lambda)+V=Y,\quad\lambda\in\Lambda.
\end{align}
Hence if we denote by $P$ the projection onto $V$, then we obtain a family of exact sequences
\[X\xrightarrow{L_\lambda}Y\xrightarrow{I_Y-P}\im(I_Y-P)\rightarrow 0,\]
and so a vector bundle $E(L,V)$ consisting of the union of the kernels of the maps $(I_Y-P)\circ L_\lambda$, $\lambda\in\Lambda$ (cf. \cite[\S III.3]{Lang}). If $\Theta(V)$ stands for the product bundle $\Lambda\times V$, then we obtain a $KO$-theory class
\begin{align}\label{defiindbund}
\ind(L):=[E(L,V)]-[\Theta(V)]\in KO(\Lambda),
\end{align}
which is called the \textit{index bundle} of $L$.\\
It is readily seen that
\begin{align}\label{dimindbund}
\dim E(L,V)=\dim(V)+\ind(L_\lambda),\quad\lambda\in\Lambda,
\end{align}
if $\Lambda$ is connected, and so $\ind(L)\in\widetilde{KO}(\Lambda)$ if and only if the operators $L_\lambda$ are of Fredholm index $0$. Moreover, if $\Lambda=\{\lambda_0\}$ is a singleton, then $KO(\{\lambda_0\})=\mathbb{Z}$ and

\[\ind(L)=\dim(E(L,V))-\dim(V)=\ind(L_{\lambda_0})\]
which shows that the definitions \eqref{defiindbund} and \eqref{Fredind} coincide in this case.\\
Let us mention for later reference the following properties of the index bundle, which can all be found in \cite{indbundleIch}:
\begin{enumerate}
\item[(i)] If $L_\lambda$ is invertible for all $\lambda\in\Lambda$, then $\ind(L)=0$.
\item[(ii)] If $\Lambda'$ is another compact topological space and $f\colon \Lambda'\rightarrow\Lambda$ a continuous map, then $f^\ast L\colon\Lambda'\rightarrow\Phi(X,Y)$ defined by $(f^\ast L)_\lambda=L_{f(\lambda)}$ is a family of Fredholm operators parametrised by $\Lambda'$ and
\begin{align}\label{naturality}
\ind(f^\ast L)=f^\ast\ind(L).
\end{align}
\item[(iii)] If $K\colon\Lambda\rightarrow\mathcal{L}(X,Y)$ is a family of compact operators, then
\[\ind(L+K)=\ind(L).\]
\item[(iv)]\label{logarithmic} If $Z$ is a Banach space and $M\colon\Lambda\rightarrow\Phi(Y,Z)$ a family of Fredholm operators, then
\[\ind(ML)=\ind(M)+\ind(L).\]
\item[(v)] If $\widetilde{X}$, $\widetilde{Y}$ are Banach spaces and $L_1\colon\Lambda\rightarrow\Phi(X,Y)$, $L_2\colon\Lambda\rightarrow\Phi(\widetilde{X},\widetilde{Y})$ are two families of Fredholm operators, then
\[\ind(L_1\oplus L_2)=\ind(L_1)+\ind(L_2)\in KO(\Lambda).\]
\end{enumerate}

Let us now denote by $[\Lambda,\Phi(X,Y)]$ the homotopy classes of all maps $L:\Lambda\rightarrow\Phi(X,Y)$. By the homotopy invariance of the index bundle, it follows that we actually obtain a well defined map

\[\ind:[\Lambda,\Phi(X,Y)]\rightarrow KO(\Lambda).\]
Atiyah and Jänich showed independently that this map is a bijection if $X=Y$ is a separable Hilbert space $H$. Their argument is as follows: they prove that the sequence

\[0\rightarrow[\Lambda,GL(H)]\hookrightarrow[\Lambda,\Phi(H)]\xrightarrow{\ind} KO(\Lambda)\rightarrow 0\]
is exact, and then use Kuiper's theorem, which states that $GL(H)$ is contractible, to conclude that $\ind$ must indeed be a bijection in this case.\\
The validity of the Atiyah-Jänich Theorem has been investigated e.g. in \cite{Banach Bundles} (cf. also \cite{indbundleIch}), however, the best possible result for general Banach spaces $E$ is the exactness of the sequence

\begin{align}\label{exactBanach}
0\rightarrow[\Lambda,GL(E)]\hookrightarrow[\Lambda,\Phi(E)]\xrightarrow{\ind} KO(\Lambda),
\end{align}
which contains much less information on how much the homotopy classes $[\Lambda,\Phi(E)]$ are classified by $KO(\Lambda)$. In particular, $[\Lambda,GL(E)]$ is not trivial in general (cf. e.g. \cite{Banach Bundles}).  We now want to point out that for $E=c_0(\mathbb{R}^N)$ the situation is as good as for separable Hilbert spaces.

\begin{prop}
The index bundle induces a bijection

\[\ind:[\Lambda,\Phi(c_0(\mathbb{R}^N))]\rightarrow KO(\Lambda).\]
\end{prop}

\begin{proof}
We only prove the assertion for $N=1$ and leave it to the reader to show that $c_0(\mathbb{R})$ is isomorphic to $c_0(\mathbb{R}^N)$ for any $N\in\mathbb{N}$, which shows the assertion in the general case.\\
We note at first that Arlt proved in \cite{Arlt} that $c_0(\mathbb{R})$ is a \textit{Kuiper space}, i.e. $GL(c_0(\mathbb{R}))$ is contractible as a topological space. Hence $[\Lambda,GL(c_0(\mathbb{R}))]$ is trivial and so it remains to show that $\ind:[\Lambda,\Phi(c_0(\mathbb{R}))]\rightarrow KO(\Lambda)$ is surjective. To this aim, let us introduce a Schauder basis $\{e_k\}$ of $c_0(\mathbb{R})$ by setting for $k\in\mathbb{N}$

\begin{align}\label{basisc0}
e_0=\{\delta_{n,0}\}_{n\in\mathbb{Z}},\qquad e_{2k}=\{\delta_{n,k}\}_{n\in\mathbb{Z}},\qquad e_{2k+1}=\{\delta_{n,-k}\}_{n\in\mathbb{Z}}.
\end{align}
Let us recall that every element in $KO(\Lambda)$ can be written in the form $[\Theta^k]-[E]$ for a bundle $E$ over $\Lambda$ and some non-negative integer $k$, where we use $\Theta^k$ to abbreviate $\Theta(\mathbb{R}^k)$. We now claim that for a given bundle $F$ over $\Lambda$ there are families $L,M:\Lambda\rightarrow\Phi(c_0(\mathbb{R}))$ such that $\ind(L)=[\Theta^k]-[\Theta^0]$ and $\ind(M)=[\Theta^0]-[F]$. Then by the logarithmic property \eqref{logarithmic}

\[\ind(LM)=\ind(L)+\ind(M)=[\Theta^k]-[F]\]
showing the surjectivity of $\ind:[\Lambda,\Phi(c_0(\mathbb{R}))]\rightarrow KO(\Lambda)$.\\
For constructing the family $L$ we let $l,r:c_0(\mathbb{R})\rightarrow c_0(\mathbb{R})$ be the left shift and the right shift with respect to the basis \eqref{basisc0}, respectively. Then, if we set $L=l^k$, the $k$-fold left shift, we obtain a constant family of surjective operators which have as kernel the space $\spann\{e_0,\ldots, e_{k-1}\}$, and consequently, $\ind(L)=[\Theta^k]-[\Theta^0]$.\\
For constructing the family $M$, we let $G$ be a vector bundle over $\Lambda$ such that $F\oplus G\cong\Theta^n$ for some $n\in\mathbb{N}$, and we let $P:\Lambda\rightarrow M(n,\mathbb{R})$ be a family of idempotent matrices such that $\im(P_\lambda)=F_\lambda$ and $\ker(P_\lambda)=G_\lambda$ for $\lambda\in\Lambda$. We set for $i\in\mathbb{N}$

\[X_i:=\spann\{e_{(i-1)n},\ldots, e_{in-1}\}\]
and define the family $M$ by

\[M_\lambda=r^n\hat P_\lambda+(I_{X_n}-\hat P_\lambda),\]
where $\hat P$ is the matrix family $P$ applied to the elements in $X_n$. Clearly, each $M_\lambda$ is injective and moreover $\im(M_\lambda)\oplus F_\lambda=c_0(\mathbb{R})$, where $F_\lambda$ is considered as a subspace of $X_1$. It is not very difficult to see that $M$ is a continuous family of bounded operators. Consequently, $M_\lambda\in\Phi(c_0(\mathbb{R}))$ and $\ind(M)=[\Theta^0]-[F]$ as we claimed.
\end{proof}
Let us now assume that $a_n\colon\Lambda\rightarrow M(N,\mathbb{R})$, $n\in\mathbb{Z}$, is a sequence of continuous families of $N\times N$ matrices such that Assumption (A3) from the previous section holds. We consider the linear operators
\begin{align}\label{L}
L_\lambda\colon c_0(\mathbb{R}^N)\rightarrow c_0(\mathbb{R}^N),\qquad (L_\lambda x)_n=x_{n+1}-a_n(\lambda)x_n,\,n\in\mathbb{Z},
\end{align}
which are easily seen to be bounded under Assumption (A3). Moreover, we obtain a continuous map $L\colon \Lambda\rightarrow\mathcal{L}(c_0(\mathbb{R}^N))$ and the aims of this section are to show that $L$ is a family of Fredholm operators and to find a formula for its index bundle $\ind(L)\in KO(\Lambda)$. Let us recall from the previous section that the families of matrices $a(\lambda,\pm\infty)$ define vector bundles $E^s(\pm\infty)$ over the parameter space $\Lambda$ if Assumption (A3) is satisfied.

\begin{theorem}\label{thm-lin}
Let $a_n\colon\Lambda\rightarrow M(N,\mathbb{R})$ be a sequence of continuous maps satisfying Assumption \emph{(A3)}. Then
\begin{itemize}
	\item[$(i)$] the operators $L_\lambda$, $\lambda\in\Lambda$, in \eqref{L} are Fredholm,
	\item[$(ii)$] the index bundle of $L$ is	
	\[\ind(L)=[E^s(+\infty)]-[E^s(-\infty)]\in KO(\Lambda).\]
\end{itemize}
\end{theorem}
\noindent
\begin{rem}\label{rem:dim}
Let us finally point out that we obtain from the previous theorem immediately the Fredholm index of the operators $L_\lambda$. Indeed, if $\iota\colon\{\lambda\}\hookrightarrow\Lambda$ denotes the canonical inclusion, then $L_\lambda=\iota^\ast L$, where we use the notation from the second property of the index bundle from above. Hence
\begin{align*}
\ind(L_\lambda)&=\iota^\ast(\ind L)=\iota^\ast([E^s(+\infty)]-[E^s(-\infty)])=\dim(E^s(+\infty))-\dim(E^s(-\infty))\\
&=\dim(E^s(\lambda,+\infty))-\dim(E^s(\lambda,-\infty))
\end{align*}
in $KO(\{\lambda\})=\mathbb{Z}$. In particular, $\ind(L)\in\widetilde{KO}(\Lambda)$ under Assumption \emph{(A4)}.
\end{rem}

%%%%%%%%%%%%%%%%%%%%%%%%%%%%%%%%%%%%%%%%%%%%%%%%%%%%%%%%%%%%%%%%%%%%%%%%%%%%%%%%%%%%%%%%%%%%%%%%%%%%%%%%%%%%%%%%%%%%%%%%%%%%%%%%%%%%%%%%%%%%%%%%%%%%%%%%%%%%%%%%%%%%%%%%%%%%%%%%%%%%%%%%%%%%%%%%%%%%%%%%%%%%%%%%%%%%%%%%%%%%%%%%%%%%%%%%%%%%%%%%%%%%%%%%%%%%%%%%%%%%%%%%%%%%%%%%%%%%%%%%%%%%%%%%%%%%%%%%%%%%%%%%%%%%%%%%%%%%%%%%%%%%%%%%%%%%%%%%%%%%%%%%%%%%%%%%%%%%%%%%%%

\subsection{Proof of Theorem \ref{thm-lin}}
We divide the proof of Theorem \ref{thm-lin} into four steps.

%%%%%%%%%%%%%%%%%%%%%%%%%%%%%%%%%%%%%%%%%%%%%%%%%%%%%%%%%%%%%%%%%%%%%%%%%%%%%%%%%%%%%%%%%%%%%%%%%%%%%%%%%%%%%%%%%%%%%%%%
\subsubsection*{Step 1: Approximation}
We define a family of matrices $\widetilde{a}_n\colon \Lambda\rightarrow M(N,\mathbb{R})$, $n\in\mathbb{Z}$, by
\begin{equation}\label{tildea}
\widetilde{a}_n(\lambda)=\begin{cases}
a(\lambda,+\infty),\quad &n\geq0\\
a(\lambda,-\infty),\quad &n< 0
\end{cases}
\end{equation}
and we let $\widetilde{L}\colon\Lambda\rightarrow\mathcal{L}(c_0(\mathbb{R}^N))$ be the family of linear operators defined by
\[(\widetilde{L}_\lambda x)_n=x_{n+1}-\widetilde{a}_n(\lambda)x_n,\quad n\in\mathbb{Z}.\]
We claim that $K_\lambda:=L_\lambda-\widetilde{L}_\lambda\in\mathcal{L}(c_0(\mathbb{R}^N))$ is compact as it is the limit of the sequence of finite rank operators $\{K^m_\lambda\}_{m\in\mathbb{N}}$ given by
\begin{align*}
(K^m_\lambda x)_n=\begin{cases}
(a_n(\lambda)-\widetilde{a}_n(\lambda))x_n,\quad &|n|\leq m\\
0,\quad &|n|>m.
\end{cases}
\end{align*}
Indeed, as $\lim\limits_{n\rightarrow\pm\infty}(a_n(\lambda)-\widetilde{a}_n(\lambda))=0$, there is for every $\varepsilon>0$ an $m_0\in\mathbb{N}$ such that
\[\|K_\lambda-K^{m_0}_\lambda\|\leq\sup_{|n|>m_0}\|a_n(\lambda)-\widetilde{a}_n(\lambda)\|<\varepsilon.\]
Moreover, it is readily seen from the definition that
\[\|K_\lambda-K^m_\lambda\|\geq\|K_\lambda-K^{m+1}_\lambda\|,\quad m\in\mathbb{N},\]
showing that $\|K_\lambda-K^m_\lambda\|<\varepsilon$ for all $m\geq m_0$.\\
As $K_\lambda$ is compact, we see that $L_\lambda$ is Fredholm if and only if $\widetilde{L}_\lambda$ is Fredholm. Moreover, $\ind(L)=\ind(\widetilde{L})\in KO(\Lambda)$ by the property (iii) of the index bundle from Section \ref{section:famthm}. Hence we can assume from now on without loss of generality that the maps $a_n\colon \Lambda\rightarrow M(N,\mathbb{R})$ in the definition of the operator $L$ are of the form \eqref{tildea}.

%%%%%%%%%%%%%%%%%%%%%%%%%%%%%%%%%%%%%%%%%%%%%%%%%%%%%%%%%%%%%%%%%%%%%%%%%%%%%%%%%%%%%%%%%%%%%%%%%%%%%%%%%%%%%%%%%%%%%%%%
\subsubsection*{Step 2: The families $L^\pm$}
We consider the closed subspaces of $c_0(\mathbb{R}^N)$ given by
\begin{align*}
X&=\{x\in c_0(\mathbb{R}^N) \mid\, x_n=0,\, n<0\},\\
Y&=\{x\in c_0(\mathbb{R}^N)\mid \, x_n=0,\, n>0\},\\
Z&=\{x\in c_0(\mathbb{R}^N)\mid \, x_n=0,\, n\geq 0\},
\end{align*}
and the bounded linear operators
\begin{align*}
L^+_\lambda\colon X\rightarrow X,\quad(L^+_\lambda x)_n=\begin{cases}
x_{n+1}-a(\lambda,\infty)x_n,\, &n\geq 0\\
0,\, &n<0
\end{cases}
\end{align*}
and
\begin{align*}
L^-_\lambda\colon Y\rightarrow Z,\quad(L^-_\lambda x)_n=\begin{cases}
0,\, &n\geq0\\
x_{n+1}-a(\lambda,-\infty)x_n,\, &n<0.
\end{cases}
\end{align*}
Note that the operators $L^\pm_\lambda$ are strictly related to $L_\lambda$, and the aim of this second step of our proof is to show that they are Fredholm. As we will see in the subsequent step, this implies the Fredholm property of $L_\lambda$ quite straightforwardly.\\
For proving that $L^\pm_\lambda$ are Fredholm, we need the following lemma.
\begin{lemma}\label{lemma-Alberto}
Let $a\in GL(N,\mathbb{R})$ be a hyperbolic matrix. Then the operator
\begin{align*}
L\colon X\rightarrow X,\quad(Lx)_n=\begin{cases}
x_{n+1}-a\,x_n,\, &n\geq0\\
0,\, &n<0
\end{cases}
\end{align*}
is surjective and
\[\ker(L)=\{x\in X \mid\, x_{n}=a^nx_0,\, n\in\mathbb{N},\, x_0\in E^s(a)\}.\]
\end{lemma}
\begin{proof}
We denote by $P^u$ the projection in $\mathbb{R}^N$ onto $E^u(a)$, by $P^s$ the projection onto $E^s(a)$, and we note that $P^u+P^s=I_N$ as $a$ is hyperbolic. We set
\begin{align*}
(Mx)_n=\begin{cases}
-\sum\limits^\infty_{k=0}a^{-1-k}P^ux_k,\, &n=0\\
\sum\limits^{n-1}_{k=0}a^{n-1-k}P^sx_k-\sum\limits^\infty_{k=n}a^{n-1-k}P^ux_k,\, &n>0\\
0,\, &n<0
\end{cases}
\end{align*}
and note that for $n=0$
\begin{align*}
(LMx)_0&=P^sx_0-\sum^\infty_{k=1}{a^{-k}P^ux_k}+\sum^\infty_{k=0}{a^{-k}P^ux_k}=P^sx_0+P^ux_0=x_0,
\end{align*}
and for $n>0$
\begin{align*}
(LMx)_n&=\sum^{n}_{k=0} a^{n-k}P^sx_k-\sum^\infty_{k=n+1}{a^{n-k}P^ux_k}-\sum^{n-1}_{k=0}a^{n-k}P^sx_k+\sum^\infty_{k=n}{a^{n-k}P^ux_k}\\
&=P^sx_n+P^ux_n=x_n,
\end{align*}
as well as $(LMx)_n=0$ for $n<0$. Hence, $LMx=x$, and in order to obtain the surjectivity of $L$, we only need to prove that $M$ maps $X$ into $X$. As $(Mx)_n=0$ for $n<0$ by definition, it remains to show that $(Mx)_n\rightarrow 0$ as $n\rightarrow\infty$. \newline\indent
For this purpose, we want to recall at first the definition of the convolution $f\ast x$ for $f\in  \ell_1(M(N,\R))\subset c_0(M(N,\R))$ and $x\in c_0(\R^N)$, where $\ell_1(M(N,\R))$ denotes the Banach space of all summable sequences $\{M_n\}_{n\in\mathbb{Z}}\subset M(N,\R)$ with respect to the usual norm on $M(N,\R)$. If $f\in \ell_1(M(N,\R))$ and $x\in c_0(\R^N)$, then $f\ast x\colon \Z\to \R^N$ is defined by
\begin{equation*}
(f\ast x)(n)=\sum_{k\in\Z} f(n-k)x_k.
\end{equation*}
Young's inequality implies that
\begin{equation*}
\|f\ast x\|_1=\sum_{n\in\Z}\|(f\ast x)(n)\|\leq \left(\sum_{k\in\Z}\|f(k)\|\right)\cdot \left(\sup_{k\in \Z} \|x_k\|\right)=\|f\|_1\cdot \|x\|
\end{equation*}
and hence $f\ast x\in \ell_1(\R^N)\subset c_0(\R^N)$.\\
We now proceed with the proof that $(Mx)_n \xrightarrow\infty 0$ as $n\rightarrow\infty$.
We define a map $f\colon \Z\to M(N,\mathbb{R})$ by

\begin{equation}\label{def-f(n)}
f(n)=a^{n-1}(\mathbbm{1}_{\N}(n)I_N-P^u).
\end{equation}
Since the spectral radius theorem for a hyperbolic matrix $a\in M(N,\R)$ implies that
\begin{align*}
\lim_{n\to\infty}\|(a|E^s(a))^n\|^{1/n}&=\alpha_s:=\max |\sigma(a|E^s(a))|<1,\\
\lim_{n\to\infty}\|(a|E^u(a))^{-n}\|^{1/n}&=\alpha_u:=\max |\sigma((a|E^u(a))^{-1})|<1,
\end{align*}
it follows that for any $\alpha\in (\max\{\alpha_s,\alpha_u\},1)$ there exists $n_0>0$ such that  for $k\geq n_0$

\begin{align*}
\left\|(a|E^s(a))^k\right\|\leq \alpha^k \text{ and } \left\|(a|E^u(a))^{-k}\right\|\leq \alpha^k.
\end{align*}
Moreover,

\begin{equation*}
a^kP^s=a^kP^sP^s=(a|E^s(a))^kP^s\quad \text{ and }\quad a^{-k}P^u=a^{-k}P^uP^u=(a|E^u(a))^{-k}P^u,
\end{equation*}
and so

\begin{align*}
\sum_{k\in\Z}\|f(k)\|&=\sum_{|k|\leq n_0}\|f(k)\|+\sum_{|k|>n_0}\|f(k)\|=\\ &=\sum_{|k|\leq n_0}\|f(k)\|+\sum_{|k|>n_0}\|a^{k-1}(\mathbbm{1}_{\N}(k)I_{N}-P^u)\|\\
&\leq\sum_{|k|\leq n_0}\|f(k)\|+\sum_{k<-n_0}\|a^{k-1}P^u\|+\sum_{k>n_0}\|a^{k-1}P^s\|\\
&=\sum_{|k|\leq n_0}\|f(k)\|+\sum_{k<-n_0}\|(a|E^u(a))^{k-1}P^u\|+\sum_{k>n_0}\|(a|E^s(a))^{k-1} P^s\|\leq\\
&\leq\sum_{|k|\leq n_0}\|f(k)\|+\sum_{k=n_0+1}^{\infty}\alpha^{k+1}\|P^u\|+\sum_{k=n_0+1}^{\infty}\alpha^{k-1}\|P^s\|\\
&=\sum_{|k|\leq n_0}\|f(k)\|+(\alpha^{n_0}+\alpha^{n_0+1})\|P^s\|+(\|P^s\|+\|P^u\|)\frac{\alpha^{n_0+2}}{1-\alpha}<\infty,
\end{align*}
which implies that $f\in \ell_1(M(N,\R))$. Finally, we observe that for $n\in\N$, $x\in X$ and $f$ from \eqref{def-f(n)}
\begin{align*}
(f\ast x)(n)&=\sum_{k\in\Z} f(n-k)x_k=\sum_{k\geq 0} f(n-k)x_k=\sum_{k\geq 0} a^{n-k-1}(\mathbbm{1}_{\N}(n-k)I_N-P^u)x_k\\
&=\sum_{k=0}^{n-1} a^{n-k-1}(1\cdot I_N-P^u)x_k+\sum_{k=n}^{\infty} a^{n-k-1}(0\cdot I_N-P^u)x_k\\
&=\sum_{k=0}^{n-1} a^{n-k-1}P^sx_k-\sum_{k=n}^{\infty} a^{n-k-1}P^ux_k=(Mx)_n.
\end{align*}
As $(f\ast x)(n)\rightarrow 0$ for $n\rightarrow\infty$ since $f\in\ell_1(M(N,\R))$, we see that indeed $Mx\in c_0(\mathbb{R}^N)$.\\
Finally, the assertion on the kernel is an immediate consequence of the definition of $L$. We only need to note that $\{a^nx_0\}$ does not converge to $0$ if $P^ux_0\neq 0$.
\end{proof}
\noindent
We see from Lemma \ref{lemma-Alberto} that $L^+_\lambda$ is surjective, and its kernel is given by
\begin{align}\label{kernelL+}
\ker(L^+_\lambda)=\{x\in X \mid\, x_{n}=a(\lambda,+\infty)^nx_0,\, n\in\mathbb{N},\, x_0\in E^s(\lambda,+\infty)\}
\end{align}
which is isomorphic to the finite dimensional space $E^s(\lambda,+\infty)$. Hence $L^+_\lambda$ is Fredholm.\\
Moreover, as $L^+_\lambda$ is surjective for all $\lambda\in\Lambda$, we see that the trivial subspace $\{0\}\subset X$ is transversal to the images of $L^+_\lambda$ as in \eqref{subspace}. Hence $\ind(L^+)=[E(L^+,\{0\})]-[\Theta(\{0\})]$, and as the fibres of $E(L^+,\{0\})$ are the kernels of $L^+_\lambda$ given by \eqref{kernelL+}, we see that the map

\[E(L^+,\{0\})\rightarrow E^s(+\infty),\quad (\lambda,\{x_n\}_{n\in\mathbb{Z}})\mapsto (\lambda,x_0)\]
is a bundle isomorphism. Consequently,

\begin{align}\label{indbundleL+}
\ind(L^+)=[E^s(+\infty)]-[\Theta(\{0\})]\in KO(\Lambda).
\end{align}
\noindent
%%%%%%%%%%%%%%%%%%%%%%%%%%%%%%%%%%%%%%%%%%%%%%%%%%%%%%%%%%%%%%%%%%%%%%%%%%%%%%%%%%%%%%%%%%%%%%%%%%%%%%%%%%%%%%%%%%%%%%%%

Our next aim is to show the Fredholm property for the operators $L^-_\lambda$, where we need a lemma that is similar to Lemma \ref{lemma-Alberto}.
\begin{lemma}\label{lemma-AlbertoII}
Let $a\in M(N,\mathbb{R})$ be a hyperbolic matrix. Then the operator
\begin{align*}
L\colon Y\rightarrow Y,\quad(Lx)_n=\begin{cases}
0,\, &n>0\\
x_{n-1}-a^{-1}\,x_n,\, &n\leq0
\end{cases}
\end{align*}
is surjective and
\[\ker(L)=\{x\in Y \mid\, x_{n}=a^nx_0,\, n<0,\, x_0\in E^u(a)\}.\]
\end{lemma}
\begin{proof}
Consider the following operators:
\begin{align*}
&\widetilde{L}\colon X\rightarrow X,\quad(\widetilde{L}x)_n=\begin{cases}
x_{n+1}-\widetilde{a}\,x_n,\, &n\geq0,\\
0,\, &n<0,
\end{cases}\\
&I\colon X\rightarrow Y,\quad(Ix)_n=x_{-n},
\end{align*}
where $\widetilde{a}:=a^{-1}$. Then $L=I\widetilde{L}I^{-1}$, and so the conclusion follows from Lemma \ref{lemma-Alberto}.
\end{proof}
We now introduce a family of operators $N_\lambda\colon Z\rightarrow Z$ by
\begin{align*}
(N_\lambda x)_n=\begin{cases}
0,\, &n\geq 0\\
-a(\lambda,-\infty)^{-1}x_n,\, &n<0,
\end{cases}
\end{align*}
and we denote by $S\colon Z\rightarrow Y$ the shift operator $(Sx)_n=x_{n-1}$. Note that $S$ and $N_\lambda$ are isomorphisms, and moreover it is readily seen that
\begin{align*}
(SN_\lambda L^-_\lambda x)_n=\begin{cases}
0,\, &n>0,\\
x_{n-1}-a(\lambda,-\infty)^{-1}x_n, &n\leq 0.
\end{cases}
\end{align*}
Consequently we infer from Lemma \ref{lemma-AlbertoII} that $SN_\lambda L^-_\lambda$ is surjective and
\begin{align}\label{kernelL-}
\ker(SN_\lambda L^-_\lambda)=\{x\in Y\mid\, x_n=a(\lambda,-\infty)^nx_0, n<0,\, x_0\in E^u(\lambda,-\infty)\}.
\end{align}
Clearly, this shows that $L^-_\lambda$ is surjective as well. Moreover, as $\ker(SN_\lambda L^-_\lambda)=\ker(L^-_\lambda)$, we see that $\ker(L^-_\lambda)$ is isomorphic to the finite dimensional space $E^u(\lambda,-\infty)$ showing that $L^-_\lambda$ is Fredholm.\\
Our previous discussion also yields the index bundle of the family $L^-\colon\Lambda\rightarrow\mathcal{L}(Y,Z)$. As $L^-_\lambda$ is surjective, we have as for $L^+$ above that the trivial space $\{0\}\subset Z$ is transversal to the image of $L^-$ and so
\[\ind(L^-)=[E(L^-,\{0\})]-[\Theta(\{0\})],\]
where the fibres of $E(L^-,\{0\})$ are the kernels of the operators $L^-_\lambda$ given by \eqref{kernelL-}. Hence we have a bundle isomorphism
\[E(L^-,\{0\})\rightarrow E^u(-\infty),\quad (\lambda,\{x_n\}_{n\in\mathbb{Z}})\mapsto (\lambda,x_0)\]
showing that
\begin{align}\label{indbundleL-}
\ind(L^-)=[E^u(-\infty)]-[\Theta(\{0\})].
\end{align}

%%%%%%%%%%%%%%%%%%%%%%%%%%%%%%%%%%%%%%%%%%%%%%%%%%%%%%%%%%%%%%%%%%%%%%%%%%%%%%%%%%%%%%%%%%%%%%%%%%%%%%%%%%%%%%%%%%%%%%%%
\subsubsection*{Step 3: Fredholm property of $L$}
We define two bounded linear operators by
\begin{align}\label{map-I}
I\colon X\oplus Z\rightarrow c_0(\mathbb{R}^N),\quad I(x,y)=x+y
\end{align}
and
\begin{align}\label{map-J}
J\colon c_0(\mathbb{R}^N)\rightarrow X\oplus Y,\quad (Jx)_n=\begin{cases}
(x_0,x_0),\, &n=0\\
(x_n,0),\, &n>0\\
(0,x_n),\, &n<0,
\end{cases}
\end{align}
and we note that $I$ is an isomorphism and $J$ is injective. Moreover, the image of $J$ is given by $\{(x,y)\in X\oplus Y\mid \, x_0=y_0\}$, which is of codimension $N$ in $X\oplus Y$. Indeed, if we let $P\colon X\oplus Y\rightarrow\mathbb{R}^N$ be the map $P(x,y)=x_0-y_0$, then we obtain an exact sequence
\[0\rightarrow c_0(\mathbb{R}^N)\xrightarrow{J} X\oplus Y\xrightarrow{P}\mathbb{R}^N\rightarrow 0\]
showing that the cokernel of $J$ is isomorphic to $\mathbb{R}^N$. Hence $J$ is Fredholm of index $-N$.\\
Let us now consider the composition $I(L^+_\lambda\oplus L^-_\lambda)J\colon c_0(\mathbb{R}^N)\rightarrow c_0(\mathbb{R}^N)$. We find that for $x\in c_0(\mathbb{R}^N)$ and $n\in\mathbb{Z}$
\begin{align*}
\begin{split}
(I(L^+_\lambda\oplus L^-_\lambda)Jx)_n&=(L^+_\lambda Jx)_n+(L^-_\lambda Jx)_n=\begin{cases}
(L^+_\lambda x)_n,\, n\geq 0\\
(L^-_\lambda x)_n,\, n<0
\end{cases}\\
&=\begin{cases}
x_{n+1}-a(\lambda,+\infty)x_n,\, n\geq 0\\
x_{n+1}-a(\lambda,-\infty)x_n,\, n<0
\end{cases}
\end{split}
\end{align*}
which is just $(L_\lambda x)_n$. Hence,

\begin{align}\label{equL}
L_\lambda=I(L^+_\lambda\oplus L^-_\lambda)J,\quad \lambda\in\Lambda,
\end{align}
and as $I,J$ and $L^\pm_\lambda$ are Fredholm operators, we infer that $L_\lambda$ is Fredholm which shows Assertion (i) of Theorem \ref{thm-lin}.

%%%%%%%%%%%%%%%%%%%%%%%%%%%%%%%%%%%%%%%%%%%%%%%%%%%%%%%%%%%%%%%%%%%%%%%%%%%%%%%%%%%%%%%%%%%%%%%%%%%%%%%%%%%%%%%%%%%%%%%%
\subsubsection*{Step 4: The index bundle}
We extend the operators $I$ and $J$ from the previous step of the proof to constant families $\mathcal{I}:\Lambda\times(X\oplus Z)\rightarrow c_0(\mathbb{R}^N)$ and $\mathcal{J}:\Lambda\times c_0(\mathbb{R}^N)\rightarrow X\oplus Y$ of Fredholm operators. Clearly, $\ind(\mathcal{I})=0$ as $I$ is an isomorphism, and $\ind(\mathcal{J})=-[\Theta(\mathbb{R}^N)]$ as we have seen in the previous step that the operator $J$ has an $n$-dimensional kernel. As

\[L=\mathcal{I}(L^+\oplus L^-)\mathcal{J}\]
by \eqref{equL}, we obtain from the properties (i) and (iv) of the index bundle, \eqref{indbundleL+} and \eqref{indbundleL-}
\begin{align*}
\ind(L)&=\ind(\mathcal{I})+\ind(L^-)+\ind(L^+)+\ind(\mathcal{J})=[E^u(-\infty)]+[E^s(+\infty)]-[\Theta(\mathbb{R}^N)]\\
&=[E^s(+\infty)]-[E^s(-\infty)],
\end{align*}
where we have used \eqref{bundlesum} in the last equality. This shows Assertion (ii) of Theorem \ref{thm-lin}.

%%%%%%%%%%%%%%%%%%%%%%%%%%%%%%%%%%%%%%%%%%%%%%%%%%%%%%%%%%%%%%%%%%%%%%%%%%%%%%%%%%%%%%%%%%%%%%%%%%%%%%%%%%%%%%%%%%%%%%%%%%%%%%%%%%%%%%%%%%%%%%%%%%%%%%%%%%%%%%%%%%%%%%%%%%%%%%%%%%%%%%%%%%%%%%%%%%%%%%%%%%%%%%%%%%%%%%%%%%%%%%%%%%%%%%%%%%%%%%%%%%%%%%%%%%%%%%%%%%%%%%%%%%%%%%%%%%%%%%%%%%%%%%%%%%%%%%%%%%%%%%%%%%%%%%%%%%%%%%%%%%%%%%%%%%%%%%%%%%%%%%%%%%%%%%%%%%%%%%%%%%

\section{Proof of Theorem \ref{thm:nonlin}}
Let $f_n\colon \Lambda\times\mathbb{R}^N\rightarrow\mathbb{R}^N$ be a sequence of continuous maps which are continuously differentiable with respect to the $\mathbb{R}^N$ variable and which satisfies the assumptions (A1) to (A5).

\begin{lemma}\label{boundedness}
Assumptions \emph{(A2)} and \emph{(A3)} imply that
\begin{equation*}
\sup_{(n,\lambda,y)\in \Z\times \Lambda\times B(0,R)}
|(Df_n)(\lambda ,y)|<\infty
\end{equation*}
for any ball $B(0,R)$ about $0\in c_0(\R^N)$ of radius $R>0$.
\end{lemma}
\begin{proof}
We first note that by Assumption (A3)
\begin{equation}\label{Assumption-A2}
C_0:=\sup_{(n,\lambda)\in \Z\times \Lambda} |
(Df_n)(\lambda,0)|<\infty.
\end{equation}
Now, fix $R>0$ and $\varepsilon>0$. Let $\delta>0$ be as in Assumption
(A2), and let $(n,\lambda,y)\in \Z\times \Lambda\times B(0,R)$. Then
there exists $n_0>0$ such that $n_0\leq M/\delta<n_0+1$.
Furthermore, there exist $0<k\leq n_0+1$ and points
$y_0=0,y_1,...,y_{k-1},y_k=y\in B(0,R)$ such that
$|y_i-y_{i+1}|<\delta$, for $i=0,...,k-1$. Thus
\begin{align*}
&|(Df_n)(\lambda,y)|\leq\\
&|(Df_n)(\lambda,0)|+|(Df_n)(\lambda,y_1)-(Df_n)(\lambda,0)|+
|(Df_n)(\lambda,y_2)-
(Df_n)(\lambda,y_1)|+\ldots+\\
&|(Df_n)(\lambda,y_{k-1})-
(Df_n)(\lambda,y_{k-2})|+|
(Df_n)(\lambda,y)-(Df_n)(\lambda,y_{k-1})|\\ &\leq C_0+k\varepsilon\leq C_0+(n_0+1)\varepsilon,
\end{align*}
where $C_0$ is as in \eqref{Assumption-A2}.
\end{proof}

We set for $x\in c_0(\mathbb{R}^N)$
\begin{align*}
(F_\lambda x)_n=x_{n+1}-f_n(\lambda,x_n),\quad n\in\mathbb{Z},
\end{align*}
and we note that by (A1) and the mean value theorem
\begin{align}\label{estimateF}
\begin{split}
\|(F_\lambda x)_n\|&\leq\|x_{n+1}\|+\|f_n(\lambda,x_n)-f_n(\lambda,0)\|\\
&\leq \|x_{n+1}\|+\left(\sup_{(n,s)\in\mathbb{Z}\times[0,1]}\|(Df_n)(\lambda,sx_n)\|\right)\,\|x_n\|,
\end{split}
\end{align}
which converges to $0$ as $n\rightarrow\pm\infty$ by Lemma \ref{boundedness}. Hence we have a map
\begin{align}\label{F}
F\colon \Lambda\times c_0(\mathbb{R}^N)\rightarrow c_0(\mathbb{R}^N),
\end{align}
and in the next section we investigate its continuity and differentiability.

%%%%%%%%%%%%%%%%%%%%%%%%%%%%%%%%%%%%%%%%%%%%%%%%%%%%%%%%%%%%%%%%%%%%%%%%%%%%%%%%%%%%%%%%%%%%%%%%%%%%%%%%%%%%%%%%%%%%%%%%

\subsection{Analytic Properties of $F$}
The aim of this section is to prove that the map $F$ in \eqref{F} is continuous, differentiable in the second variable and $(DF_\lambda)(x)\in\mathcal{L}(c_0(\mathbb{R}^N))$ depends continuously on $(\lambda,x)\in\Lambda\times c_0(\mathbb{R}^N)$. Our argument mainly follows \cite[Lemma 2.3]{Poetzsche} and \cite[Lemma 6.1]{JacoboRobertI}.\\
We note at first that given $x\in c_0(\mathbb{R}^N)$ and $\lambda\in\Lambda$, there is a closed ball $\overline{B}$ of finite radius about $0\in\mathbb{R}^N$ such that $x_n$ is in the interior of $\overline{B}$ for all $n\in\mathbb{Z}$. Hence by (A2), for every $\varepsilon>0$ there is $\delta>0$ such that $y_n\in B$ if $\|x_n-y_n\|<\delta$, and
\begin{align}\label{Fcont}
\|F(\lambda,x)-F(\mu,y)\|=\sup_{n\in\mathbb{Z}}\|f_n(\lambda,x_n)-f_n(\mu,y_n)\|<\varepsilon
\end{align}
if $d(\lambda,\mu)+\sup_{n\in\mathbb{Z}}\|x_n-y_n\|<\delta$. Hence $F$ is continuous at any point $(\lambda,x)\in \Lambda\times c_0(\mathbb{R}^N)$.\\
We now introduce a map $T:\Lambda\times c_0(\mathbb{R}^N)\rightarrow\mathcal{L}(c_0(\mathbb{R}^N))$ by
\[((T_\lambda x)y)_n=y_{n+1}-(Df_n)(\lambda,x_n)y_n,\quad n\in\mathbb{Z},\]
and note that $T$ is well defined by Lemma \ref{boundedness}. Moreover, its continuity follows by (A2) as in \eqref{Fcont}. Our aim is to show that $(DF_\lambda)(x)=T_\lambda(x)$ for every fixed $\lambda\in\Lambda$ and $x\in c_0(\mathbb{R}^N)$. As $T$ is continuous, this shows that $F$ is differentiable in the second variable and $(DF_\lambda)(x)$ depends continuously on $(\lambda,x)$.\\
We obtain for $h\in c_0(\mathbb{R}^N)$ and $\lambda\in\Lambda$
\begin{align*}
r_\lambda(x,h)&=\|F_\lambda(x+h)-F_\lambda(x)-T_\lambda(x)h\|=\sup_{n\in\mathbb{Z}}\|f_n(\lambda,x_n+h_n)-f_n(\lambda,x_n)-(Df_n)(\lambda,x_n)h_n\|\\
&=\sup_{n\in\mathbb{Z}}\left\|\int^1_0{(Df_n)(\lambda,x_n+sh_n)h_n\,ds-(Df_n)(\lambda,x_n)h_n}\right\|\\
&\leq\sup_{n\in\mathbb{Z}}\left(\int^1_0{\left\|(Df_n)(\lambda,x_n+sh_n)-(Df_n)(\lambda,x_n)\right\|ds}\right)\,\sup_{n\in\mathbb{Z}}\|h_n\|\\
&\leq \|h\|\int^1_0{\sup_{n\in\mathbb{Z}}\left\|(Df_n)(\lambda,x_n+sh_n)-(Df_n)(\lambda,x_n)\right\|ds}\\
&\leq\|h\|\sup_{s\in[0,1]}\sup_{n\in\mathbb{Z}}\left\|(Df_n)(\lambda,x_n+sh_n)-(Df_n)(\lambda,x_n)\right\|.
\end{align*}
Now it follows from (A2) that
\[0\leq\frac{r_\lambda(x,h)}{\|h\|}\leq\sup_{s\in[0,1]}\sup_{n\in\mathbb{Z}}\left\|(Df_n)(\lambda,x_n+sh_n)-(Df_n)(\lambda,x_n)\right\|\rightarrow 0 \]
as $h\rightarrow 0$ in $c_0(\mathbb{R}^N)$, and so indeed $(DF_\lambda)(x)=T_\lambda(x)$, $x\in c_0(\mathbb{R}^N)$.

%%%%%%%%%%%%%%%%%%%%%%%%%%%%%%%%%%%%%%%%%%%%%%%%%%%%%%%%%%%%%%%%%%%%%%%%%%%%%%%%%%%%%%%%%%%%%%%%%%%%%%%%%%%%%%%%%%%%%%%%
\subsection{Finite dimensional reduction}\label{proof-indbund}
We note at first that by the results of the previous section
\[((DF_\lambda)(0)x)_n=x_{n+1}-a_n(\lambda)x_n,\quad n\in\mathbb{Z},\]
where $a_n(\lambda)=(Df_n)(\lambda,0)$ as introduced in Section \ref{section-mainthm}. By (A3) and Theorem \ref{thm-lin}, these operators are Fredholm. Moreover, as $\dim E^s(\lambda,+\infty)=\dim E^s(\lambda,-\infty)$ by (A4), we see from Remark \ref{rem:dim} that the Fredholm index of $(DF_\lambda)(0)$ vanishes, i.e. $(DF_\lambda)(0)\in\Phi_0(c_0(\mathbb{R}^n))$ for all $\lambda\in\Lambda$.\\
In what follows, we set $L_\lambda:=(DF_\lambda)(0)$ and so $L=\{L_\lambda\}$ is a family in $\Phi_0(c_0(\mathbb{R}^N))$. Let now $V\subset c_0(\mathbb{R}^N)$ be a subspace which is transversal to the image of $L$ as in \eqref{subspace}, i.e.
\begin{align}\label{proof-transveralL}
\im(L_\lambda)+V=c_0(\mathbb{R}^N),\quad\lambda\in\Lambda,
\end{align}
and so

\[\ind(L)=[E(L,V)]-[\Theta(V)]\in\widetilde{KO}(\Lambda).\]
From now on we simplify our notation by setting $E:=E(L,V)$.\\
Let $P_V$ be a projection onto the finite dimensional space $V$ and set $W=\ker(P_V)$ as well as $P_W=(I_{c_0(\mathbb{R}^N)}-P_V)$. If we use that $c_0(\mathbb{R}^N)=V\oplus W$, then we can write $F$ as $F=(F^1,F^2)$, where
\[F^1=P_VF\colon \Lambda\times c_0(\mathbb{R}^N)\rightarrow V,\quad F^2=P_WF\colon \Lambda\times c_0(\mathbb{R}^N)\rightarrow W.\]
As $E$ is a finite dimensional subbundle of $\Theta(c_0(\mathbb{R}^N))$, there is a family $P\colon\Lambda\rightarrow\mathcal{L}(c_0(\mathbb{R}^N))$ of projections such that $\im(P_\lambda)=E_\lambda$ for all $\lambda\in\Lambda$ (cf. e.g. \cite{indbundleIch}). Considering $P$ as a bundle morphism between $\Theta(c_0(\mathbb{R}^N))$ and $E$, we can define a fibre preserving map
\[\phi\colon \Theta(c_0(\mathbb{R}^N))\rightarrow\Theta(W)\oplus E,\quad \phi(\lambda,x)=(\lambda,F^2(\lambda,x),P_\lambda x)\]
which maps the zero section $\Lambda\times\{0\}$ of $\Theta(c_0(\mathbb{R}^N))$ to the zero section of $\Theta(W)\oplus E$. The main step in our finite dimensional reduction is the following technical lemma.

\begin{lemma}\label{lemma-reduction}
There are open neighbourhoods $\Omega_1\subset\Theta(c_0(\mathbb{R}^N))$ and $\Omega_2\subset\Theta(W)\oplus E$ of the zero sections such that $\phi:\Omega_1\rightarrow\Omega_2$ is a homeomorphism.
\end{lemma}

\begin{proof}
We split the proof into three parts.

%%%%%%%%%%%%%%%%%%%%%%%%%%%%%%%%%%%%%%%%%%%%%%%%%%%%%%%%%%%%%%%%%%%%%%%%%%%%%%%%%%%%%%%%%%%%%%%%%%%%%%%%%%%%%%%%%%%%%%%%

\subsubsection*{Step 1: From $\phi$ to $\widetilde{\phi}$}
Let us note at first that each $\phi_\lambda:=\phi(\lambda,\cdot):c_0(\mathbb{R}^N)\rightarrow W\oplus E_\lambda$ is differentiable and

\[(D\phi_\lambda)(0)=((DF^2_\lambda)(0),P_\lambda)\in\mathcal{L}(c_0(\mathbb{R}^N),W\oplus E_\lambda).\]
As the map $(DF^2_\lambda)(0)=P_WDF_\lambda(0)=P_WL_\lambda$ is surjective by \eqref{proof-transveralL}, $P_\lambda$ is surjective as a projection and $\ker((DF^2_\lambda)(0))=\im(P_\lambda)$, we see that $(D\phi_\lambda)(0)$ is an isomorphism.\\
If now $(\lambda_0,0)\in\Theta(c_0(\mathbb{R}^N))$ is given, then there is a neighbourhood $U$ of $\lambda_0$ and a trivialisation $\tau:E\mid_U\rightarrow U\times E_{\lambda_0}$. As $E_{\lambda_0}=E(L,V)_{\lambda_0}$ and $\dim E(L,V)_{\lambda_0}=\ind L_{\lambda_0}+\dim V=\dim V$ by \eqref{dimindbund}, we obtain a bundle isomorphism
\[\rho\colon \Theta(W)\oplus E\rightarrow \Theta(W)\oplus\Theta(V)\cong\Theta(c_0(\mathbb{R}^N))\]
over $U$, which is the identity on $\Theta(W)$. We now compose the map $\rho\circ\phi$ on the right with $(D\phi_{\lambda_0}(0))^{-1}\rho^{-1}_{\lambda_0}$, and obtain a a fibre preserving map $\widetilde{\phi}$ between neighbourhoods in $\Theta(c_0(\mathbb{R}^N))$ of $(\lambda_0,0)$ such that
\begin{align}\label{diffId}
(D\widetilde{\phi}_{\lambda_0})(0)=I_{c_0(\mathbb{R}^N)}.
\end{align}

%%%%%%%%%%%%%%%%%%%%%%%%%%%%%%%%%%%%%%%%%%%%%%%%%%%%%%%%%%%%%%%%%%%%%%%%%%%%%%%%%%%%%%%%%%%%%%%%%%%%%%%%%%%%%%%%%%%%%%%%

\subsubsection*{Step 2: $\widetilde{\phi}$ is a local homeomorphism}
Here and subsequently, we need the following folklore result which we recall for the reader's convenience.
\begin{lemma}\label{FPunique}
Let $X$ be a Banach space and $f:\Lambda\times X\rightarrow X$ a continuous map such that
\[\|f_\lambda(x)-f_\lambda(y)\|<q\|x-y\|,\quad x,y\in X,\, \lambda\in\Lambda,\]
for some $q\in(0,1)$. Then the map
\[\Lambda\rightarrow X, \,\lambda\mapsto x_\lambda\]
which assigns to $\lambda\in\Lambda$ the unique fixed-point $x_\lambda\in X$ of $f_\lambda$ is continuous.
\end{lemma}
\begin{proof}
For a given $\lambda_0\in\Lambda$, it is readily seen from the contraction inequality that
\begin{align*}
\|f^{(k)}_\lambda(x_{\lambda_0})-x_{\lambda_0}\|\leq\left(\sum^{k-1}_{j=0}{q^j}\right)\|f_\lambda(x_{\lambda_0})-x_{\lambda_0}\|,
\end{align*}
where $f^{(k)}_\lambda(x_{\lambda_0})$ means the $k$-fold application of $f_\lambda$ to $x_{\lambda_0}$. As $f^{(k)}_\lambda(x_{\lambda_0})\rightarrow x_\lambda$, $k\rightarrow\infty$, and $f_{\lambda_0}(x_{\lambda_0})=x_{\lambda_0}$, we obtain
\[\|x_\lambda-x_{\lambda_0}\|\leq\frac{1}{1-q} \|f_\lambda(x_{\lambda_0})-f_{\lambda_0}(x_{\lambda_0})\|\]
showing that the fixed-points depend continuously on the parameter.
\end{proof}
Now we proceed with the construction of the neighbourhoods $\Omega_1$ and $\Omega_2$. By \eqref{diffId} there is a neighbourhood $\widetilde{U}\subset \lambda$ of $\lambda_0$ and $\delta>0$ such that
\begin{align}\label{fixedpoint}
\|I_{c_0(\mathbb{R}^{N})}-(D\widetilde{\phi}_\lambda)(x)\|\leq\frac{1}{2},\quad x\in B(0,\delta),\, \lambda\in\widetilde{U}.
\end{align}
We now define for $\lambda\in\widetilde{U}$ and $y\in B(0,\frac{1}{2}\delta)$ a map

\[c_{(\lambda,y)}:c_0(\mathbb{R}^N)\rightarrow c_0(\mathbb{R}^N),\qquad c_{(\lambda,y)}(x)=x-\widetilde{\phi}_\lambda(x)+y.\]
It follows from \eqref{fixedpoint}, $\widetilde{\phi}_\lambda(0)=0$, and the mean value theorem that for all $x$ in the closed ball $ \overline{B}(0,\delta)$
\[\|c_{(\lambda,y)}(x)\|\leq\|x-\widetilde{\phi}_\lambda(x)\|+\|y\|\leq\frac{1}{2}\|x\|+\|y\|<\delta\]
and so $c_{(\lambda,y)}$ maps $\overline{B}(0,\delta)$ into itself. Moreover,
\[\|c_{(\lambda,y)}(x)-c_{(\lambda,y)}(\overline{x})\|\leq \frac{1}{2}\|x-\overline{x}\|,\quad x,\overline{x}\in\overline{B}(0,\delta),\]
and consequently $c_{(\lambda,y)}$ has for any ${(\lambda,y)}\in\widetilde{U}\times B(0,\frac{1}{2}\delta)$ a unique fixed point $x_{(\lambda,y)}$ which depends continuously on ${(\lambda,y)}$ by Lemma \ref{FPunique}. Hence $\widetilde{\phi}_\lambda\colon B(0,\delta)\rightarrow B(0,\frac{1}{2}\delta)$, $\lambda\in\widetilde{U}$, is a homeomorphism, where $\widetilde{\phi}^{-1}_\lambda(y)=x_{(\lambda,y)}$. As the fixed-point $x_{(\lambda,y)}$ depends continuously on $(\lambda,y)\in \widetilde{U}\times B(0,\delta)$ we see that also the fibre-preserving map
\[\widetilde{\phi}^{-1}\colon\widetilde{U}\times B(0,\frac{1}{2}\delta)\rightarrow\widetilde{U}\times B(0,\delta),\quad\widetilde{\phi}^{-1}(\lambda,y)=(\lambda,x_{(\lambda,y)})\]
is continuous.

\subsubsection*{Step 3: $\phi$ is a homeomorphism}
As $\widetilde{\phi}$ and $\phi$ only differ by bundle isomorphisms, we conclude that every point $(\lambda_0,0)\in\Lambda\times c_0(\mathbb{R}^N)$ has a neighbourhood in $\Theta(c_0(\mathbb{R}^N))$ which is mapped by $\phi$ homeomorphically onto a neighbourhood of $(\lambda_0,0)$ in $\Theta(W)\oplus E$.\\
If we now cover $\Lambda\times\{0\}$ by such neighbourhoods and call their unions $\Omega_1$ and $\Omega_2$, respectively, then $\phi$ is a local homeomorphism on the obtained set $\mathcal{O}_1$. Moreover, as $\phi$ is fibre-preserving, it is injective on $\mathcal{O}_1$ and so it is a homeomorphism onto $\Omega_2=\phi(\Omega_1)$.
\end{proof}
\noindent
We now assume:

\begin{enumerate}
\item[(A)] There is no bifurcation point of the map $F:\Lambda\times c_0(\mathbb{R}^N)\rightarrow c_0(\mathbb{R}^N)$.
\end{enumerate}
By Lemma \ref{lemma-reduction} there are open neighbourhoods $\Omega_1$ and $\Omega_2$ of $\Lambda\times\{0\}$ in $\Theta(c_0(\mathbb{R}^N))$ and $\Theta(W)\oplus E$, respectively, such that $\phi:\Omega_1\rightarrow\Omega_2$ is a fibre preserving homeomorphism. Moreover, as $F$ has no bifurcation point, we can assume that

\begin{align}\label{Fneq0}
F(\lambda,x)\neq 0,\quad (\lambda,x)\in\Omega_1,\,\, x\neq 0.
\end{align}
The intersection $E\cap\Omega_2$ is an open neighbourhood of the zero section of $E$ and

\[\psi\colon E\cap\Omega_2\rightarrow\Theta(c_0(\mathbb{R}^N)),\quad \psi(v)=\phi^{-1}(0,v)\]
is a homeomorphism onto $\psi(E\cap \Omega_1)\subset\Omega_2$. Moreover, it follows from the definition of $\phi$ that

\begin{align}\label{imH}
F^2_\lambda((\psi_\lambda(v)))=F^2_\lambda(\varphi^{-1}_\lambda(0,v))=0\quad \text{for}\,\, v\in (E\cap\Omega_2)_\lambda.
\end{align}
Let us now consider again the map $F^1=P_VF\colon \Lambda\times c_0(\mathbb{R}^N)\rightarrow V$. As $F^1_\lambda(\psi_\lambda(v))=0$ by \eqref{imH}, and $F_\lambda(\psi_\lambda(v))\neq 0$ by \eqref{Fneq0}, we see that

\begin{align}\label{Gneq0}
G_\lambda(\psi_\lambda(v))\neq 0\quad\text{for all}\quad v\in (E\cap\Omega_2)_\lambda,\quad v\neq 0\in E_\lambda.
\end{align}

\subsection{Sphere bundles and Dold's Theorem}
Let us recall that we assume in Theorem \ref{thm:nonlin} the existence of some $\lambda_0\in\Lambda$ such that $L_{\lambda_0}=(DF_{\lambda_0})(0)$ is invertible. Hence, by the inverse function theorem, there is a neighbourhood $\mathcal{O}$ of $0$ in $c_0(\mathbb{R}^N)$ on which $F_{\lambda_0}$ is a diffeomorphism onto its image. Using the compactness of $\Lambda$, we now let $r>0$ be such that

\begin{enumerate}
	\item[(i)] the closed disc bundle $D(E,r)$ in $E$ is contained in $E\cap\Omega_2$,
	\item[(ii)] $\psi_{\lambda_0}(D(E_{\lambda_0},r))\subset\mathcal{O}$.
\end{enumerate}
In what follows, we denote by $S(E,r)$ the associated sphere bundle to $D(E,r)$ in $E$.\\
We now let $S^{n-1}$ be the unit sphere in the $n$-dimensional space $V$, and we obtain a fibre bundle map

\[\Gamma\colon S(E,r)\rightarrow\Lambda\times S^{n-1},\quad \Gamma(v)=(\pi(v),\|F^1(\psi(v))\|^{-1}\, F^1(\psi(v))),\]
where $\pi\colon E\rightarrow\Lambda$ denotes the bundle projection. It follows from \eqref{Gneq0}, (i) and (ii) that the restriction of $F^1$ to $\psi_{\lambda_0}(D(E_{\lambda_0},r))$ is a diffeomorphism onto its image, which shows that $\Gamma_{\lambda_0}\colon S(E,r)_{\lambda_0}\rightarrow S^{n-1}$ is a homotopy equivalence.\\
Let us now recall the following classical theorem that was proved by Albrecht Dold in \cite{Dold} (cf. \cite{Crabb}):
\begin{theorem}
Let $f\colon\zeta\rightarrow\eta$ be a fibre preserving map between two fibre bundles over the connected compact CW-complex $\Lambda$. Then $f$ is a fibrewise homotopy equivalence if and only if $f_{\lambda_0}\colon\zeta_{\lambda_0}\rightarrow\eta_{\lambda_0}$ is a homotopy equivalence for some $\lambda_0\in\Lambda$.
\end{theorem}
\noindent
As $\Gamma_{\lambda_0}$ is a homotopy equivalence, we see by Dold's Theorem that $\Gamma\colon S(E,r)\rightarrow S^{n-1}$ is a fibrewise homotopy equivalence. Hence we obtain from the definition of the $J$-homomorphism that
\[J(\ind L)=J(E)=0\in J(\Lambda).\]
However, by Theorem \ref{thm-lin}, $\ind(L)=[E^s(+\infty)]-[E^s(-\infty)]$ showing that
\[J(E^s(+\infty))=J(E^s(-\infty))\in J(\Lambda),\]
which contradicts our assumption of Theorem \ref{thm:nonlin}. Consequently, our assumption (A) that $F$ does not have a bifurcation point is wrong and so Theorem \ref{thm:nonlin} is proved.

%%%%%%%%%%%%%%%%%%%%%%%%%%%%%%%%%%%%%%%%%%%%%%%%%%%%%%%%%%%%%%%%%%%%%%%%%%%%%%%%%%%%%%%%%%%%%%%%%%%%%%%%%%%%%%%%%%%%%%%%%%%%%%%%%%%%%%%%%%%%%%%%%%%%%%%%%%%%%%%%%%%%%%%%%%%%%%%%%%%%%%%%%%%%%%%%%%%%%%%%%%%%%%%%%%%%%%%%%%%%%%%%%%%%%%%%%%%%%%%%%%%%%%%%%%%%%%%%%%%%%%%%%%%%%%%%%%%%%%%%%%%%%%%%%%%%%%%%%%%%%%%%%%%%%%%%%%%%%%%%%%%%%%%%%%%%%%%%%%%%%%%%%%%%%%%%%%%%%%%%%%

\section{Proof of Theorem \ref{thm:nonlinII}}
As in the assertion of Theorem \ref{thm:nonlinII}, we denote by $B$ the set of all bifurcation points of \eqref{nonlin}. In what follows, we use without further reference the fact that $\dim(B)\geq k$ if $\check{H}^k(B;G)\neq 0$ for some abelian coefficient group $G$ (cf. \cite[VIII.4.A]{Hurewicz}), where $\check{H}^k(B;G)$ denotes the $k$-th \v{C}ech cohomology group. We now divide the proof into two steps depending on whether or not $\Lambda\setminus B$ is connected.
\subsubsection*{Step 1: Proof of Theorem \ref{thm:nonlinII} if $\Lambda\setminus B$ is not connected}
If $\Lambda\setminus B$ is not connected, then the reduced singular homology group $\tilde{H}_0(\Lambda \setminus B;\mathbb{Z}_2)$ is non-trivial. As $\Lambda$ is connected, the long exact sequence of reduced homology (cf. \cite[\S IV.6]{Bredon})
\begin{align*}
\ldots\rightarrow H_1(\Lambda,\Lambda\setminus B;\mathbb{Z}_2)\rightarrow\widetilde{H}_0(\Lambda\setminus B;\mathbb{Z}_2)\rightarrow\widetilde{H}_0(\Lambda;\mathbb{Z}_2)=0,
\end{align*}
shows that there is a surjective map
\[H_1(\Lambda,\Lambda\setminus B;\mathbb{Z}_2)\rightarrow\widetilde{H}_0(\Lambda\setminus B;\mathbb{Z}_2),\]
and so we see that $H_1(\Lambda,\Lambda\setminus B;\mathbb{Z}_2)$ is non-trivial.\\
It is an immediate consequence of Definition \ref{def:bif} that the set $B$ is closed. Hence we can apply Poincar\'e-Lefschetz duality (cf. \cite[Cor. VI.8.4]{Bredon}) to obtain an isomorphism
\begin{align*}
H_1(\Lambda,\Lambda\setminus B;\mathbb{Z}_2)\xrightarrow{\cong} \check{H}^{k-1}(B;\mathbb{Z}_2),
\end{align*}
which implies that $\check{H}^{k-1}(B;\mathbb{Z}_2)\neq 0$. Consequently, as $k\geq 2$, $B$ is not contractible to a point and, moreover, we obtain $\dim B\geq k-1$ which is greater or equal to $k-i$.

%%%%%%%%%%%%%%%%%%%%%%%%%%%%%%%%%%%%%%%%%%%%%%%%%%%%%%%%%%%%%%%%%%%%%%%%%%%%%%%%%%%%%%%%%%%%%%%%%%%%%%%%%%%%%%%%%%%%%%%%

\subsubsection*{Step 2: Proof of Theorem \ref{thm:nonlinII} if $\Lambda\setminus B$ is connected}
We denote by
\begin{align*}
\langle\cdot,\cdot\rangle:H^{i}(\Lambda;\mathbb{Z}_2)\times H_{i}(\Lambda;\mathbb{Z}_2)\rightarrow\mathbb{Z}_2
\end{align*}
the duality pairing which is non-degenerate as $\mathbb{Z}_2$ is a field. As $w_i(E^s(+\infty))\neq w_i(E^s(-\infty))$ by assumption, there is some $\alpha\in H_{i}(\Lambda;\mathbb{Z}_2)$ such that
\begin{align}\label{alpha}
\langle w_i(E^s(+\infty)),\alpha\rangle\neq\langle w_i(E^s(-\infty)),\alpha\rangle.
\end{align}
Let now $\eta\in\check{H}^{k-i}(\Lambda;\mathbb{Z}_2)$ be the Poincar\'e dual of $\alpha$ with respect to a fixed $\mathbb{Z}_2$-orientation of $\Lambda$. According to \cite[Cor. VI.8.4]{Bredon}, there is a commutative diagram
\begin{align*}
\xymatrix{&\check{H}^{k-i}(\Lambda;\mathbb{Z}_2)\ar[r]^{\iota^\ast}&\check{H}^{k-i}(B;\mathbb{Z}_2)\\
H_{i}(\Lambda\setminus B;\mathbb{Z}_2)\ar[r]^{j_\ast}&H_{i}(\Lambda;\mathbb{Z}_2)\ar[u]\ar[r]^(.4){\pi_\ast}&H_{i}(\Lambda,\Lambda\setminus B;\mathbb{Z}_2)\ar[u]
}
\end{align*}
where the lower horizontal sequence is part of the long exact homology sequence of the pair $(\Lambda,\Lambda\setminus B)$ and the vertical arrows are isomorphisms given by Poincar\'e-Lefschetz duality. By commutativity, the class $\iota^\ast\eta$ is dual to $\pi_\ast\alpha$, and we now assume by contradiction that $\pi_\ast\alpha$ is trivial.\\
By exactness of the lower horizontal sequence, there is $\beta\in H_{i}(\Lambda\setminus B;\mathbb{Z}_2)$ such that $\alpha=j_\ast\beta$. Moreover, as homology is compactly supported (cf. \cite[Sect. 20.4]{May}), there is a compact connected CW-complex $P$ and a map $g:P\rightarrow\Lambda\setminus B$ such that $\beta=g_\ast\gamma$ for some $\gamma\in H_{i}(P;\mathbb{Z}_2)$.\\
By (A5) there is some $\lambda_0\in\Lambda$ such that the difference equation \eqref{equlin} has only the trivial solution in $c_0(\mathbb{R}^N)$. We now consider again the map $F\colon\Lambda\times c_0(\mathbb{R}^N)\rightarrow c_0(\mathbb{R}^N)$, which we introduced in \eqref{F}, and recall that $L_\lambda=(DF_\lambda)(0)$ is Fredholm of index $0$ and its kernel is given by all solutions of \eqref{equlin} in $c_0(\mathbb{R}^N)$. We see that $L_{\lambda_0}$ is invertible, and so we obtain from the implicit function theorem that the only solutions of $F(\lambda,x)=0$ in a neighbourhood of $(\lambda_0,0)$ are of the form $(\lambda,0)$. Consequently, $\lambda_0\notin B$.\\
As $\Lambda\setminus B$ is connected, there is a path joining $\lambda_0$ to $g(p_0)$ for a $0$-cell $p_0$ of $P$. After attaching a $1$-cell to $p_0$, we can deform $g$ such that $\lambda_0$ belongs to its image. This does not affect the property that $\beta=g_\ast\gamma$ and so we can assume without loss of generality that $\lambda_0\in\im(g)$.\\
We now set $\overline{g}=j\circ g:P\rightarrow\Lambda$ and consider the family of discrete dynamical systems
\begin{align}\label{fP}
x_{n+1}=\overline{f}_n(p,x_n),\quad n\in\mathbb{Z},
\end{align}
for $x\in c_0(\mathbb{R}^N)$ which is parametrised by $P$, where $\overline{f}:P\times\mathbb{R}^N\rightarrow\mathbb{R}^N$ is defined by $\overline{f}_n(p,u)=f_n(\overline{g}(p),u)$. Clearly, $\overline{a}_n(p):=(D\overline{f}_n)(p,0)=a_n(\overline{g}(p))$, $n\in\mathbb{Z}$, and as the stable subspaces $\overline{E}^s(p,\pm\infty)$ of $\{\overline{a}_n(p)\}_{n\in\mathbb{Z}}$ are
\[\overline{E}^s(p,\pm\infty)=\{u\in\mathbb{R}^N:\,u\in E(\overline{g}(p),\pm\infty)\},\quad p\in P,\]
the corresponding stable bundles at $\pm\infty$ are given by the pullbacks

\begin{align}\label{pullbacks}
\overline{E}^s(+\infty)=\overline{g}^\ast(E^s(+\infty)),\quad \overline{E}^s(-\infty)=\overline{g}^\ast(E^s(-\infty)).
\end{align}
Moreover, as $\lambda_0$ is in the image of $g$, there is some $p_0\in P$ such that
\[x_{n+1}=\overline{a}_n(p_0)x_n,\quad n\in\mathbb{Z},\]
has only the trivial solution. Of course, $\overline{g}$ sends bifurcation points of \eqref{fP} to bifurcation points of \eqref{nonlin}, and as $\overline{g}(P)\cap B=\emptyset$, we see that the family \eqref{fP} has no bifurcation points. Consequently,
\[J(\overline{E}^s(+\infty))=J(\overline{E}^s(-\infty))\in J(P)\]
by Theorem \ref{thm:nonlin} showing that
\[w_i(\overline{E}^s(+\infty)))=w_i(\overline{E}^s(-\infty)))\in H^{i}(P;\mathbb{Z}_2).\]
By \eqref{pullbacks}, we obtain
\begin{align*}
0&=\langle w_i(\overline{E}^s(+\infty))-w_i(\overline{E}^s(-\infty)),\gamma\rangle=\langle \overline{g}^\ast w_i(E^s(+\infty))-\overline{g}^\ast w_i(E^s(-\infty)),\gamma\rangle\\
&=\langle \overline{g}^\ast(w_i(E^s(+\infty))- w_i(E^s(-\infty))),\gamma\rangle=\langle w_i(E^s(+\infty))- w_i(E^s(-\infty)),\overline{g}_\ast\gamma\rangle\\
&=\langle w_i(E^s(+\infty))- w_i(E^s(-\infty)),j_\ast g_\ast\gamma\rangle=\langle w_i(E^s(+\infty))- w_i(E^s(-\infty)),j_\ast\beta\rangle\\
&=\langle w_i(E^s(+\infty))- w_i(E^s(-\infty)),\alpha\rangle
\end{align*}
which is a contradiction to \eqref{alpha}.\\
Consequently, $\pi_\ast\alpha$ and so $\iota^\ast\eta\in\check{H}^{k-i}(B;\mathbb{Z}_2)$ is non-trivial. This shows that $\dim B\geq k-i$, and moreover $B$ is not contractible to a point as $k-i\geq 1$.

%%%%%%%%%%%%%%%%%%%%%%%%%%%%%%%%%%%%%%%%%%%%%%%%%%%%%%%%%%%%%%%%%%%%%%%%%%%%%%%%%%%%%%%%%%%%%%%%%%%%%%%%%%%%%%%%%%%%%%%%%%%%%%%%%%%%%%%%%%%%%%%%%%%%%%%%%%%%%%%%%%%%%%%%%%%%%%%%%%%%%%%%%%%%%%%%%%%%%%%%%%%%%%%%%%%%%%%%%%%%%%%%%%%%%%%%%%%%%%%%%%%%%%%%%%%%%%%%%%%%%%%%%%%%%%%%%%%%%%%%%%%%%%%%%%%%%%%%%%%%%%%%%%%%%%%%%%%%%%%%%%%%%%%%%%%%%%%%%%%%%%%%%%%%%%%%%%%%%%%%%%

\section{An Example for $\Lambda=T^k$}
The aim of this section is to give an example of our theory, where the dimension of the dynamical systems is $N=2$ and the parameter space is the $k$-dimensional torus $T^k$ for some $k\in\mathbb{N}$.\\
In what follows we denote coordinates of the $k$-torus by $\lambda=(\lambda_1,\ldots,\lambda_k)$ and we will use without further mentioning the identification $\lambda_j=e^{i\Theta_j}$ for $\Theta_j\in(-\pi,\pi]$ and $j=1,\ldots,k$.\\
Let now $h_n:T^k\times\mathbb{R}^2\rightarrow\mathbb{R}^2$, $n\in\mathbb{Z}$, be a sequence of maps that satisfies (A1)--(A2) and
\begin{itemize}
	\item[(B1)] $(Dh_n)(\lambda,0)\rightarrow 0$ as $n\rightarrow\pm\infty$ uniformly in $\lambda\in T^k$.
\end{itemize}
Moreover, we set
\[S:=\{\lambda\in T^k:\,\Theta_1+\cdots+\Theta_k=(2l-1)\pi,\,l\in\mathbb{Z},\, -k\leq 2l-1\leq k\}\subset T^k,\]
which is a set of measure $0$, and we require
\begin{itemize}
	\item[(B2)] there is some $\lambda_0\notin S$ such that $\sup_{n\in\mathbb{Z}}\|(Dh_n)(\lambda_0,0)\|$ is sufficiently small.
\end{itemize}
Of course, this assumption may sound a bit vague, but it holds in any case if $(Dh_n)(\lambda_0,0)=0$ for all $n\in\mathbb{Z}$ at some $\lambda_0\notin S$, and we will derive a bound on $\sup_{n\in\mathbb{Z}}\|(Dh_n)(\lambda_0,0)\|$ below in the proof of Theorem \ref{thm:application}.\\
We consider the family of discrete dynamical systems
\begin{align}\label{application}
x_{n+1}=a_n(\lambda)x_n+h_n(\lambda,x_n),\quad n\in\mathbb{Z},
\end{align}
for $\lambda=(\lambda_1,\ldots,\lambda_k)\in T^k$ and
\begin{align*}
a_n(\lambda)=\begin{cases}
a(\lambda),&\quad n\geq 0\\
a(1,\ldots,1),&\quad n<0,
\end{cases}
\end{align*}
where
\begin{align*}
a(\lambda)=\begin{pmatrix}
\frac{1}{2}+\frac{3}{2}\sin^2\left(\frac{\Theta_1+\cdots+\Theta_k}{2}\right)&-\frac{3}{4}\sin(\Theta_1+\cdots+\Theta_k)\\
-\frac{3}{4}\sin(\Theta_1+\cdots+\Theta_k)&\frac{1}{2}+\frac{3}{2}\cos^2\left(\frac{\Theta_1+\cdots+\Theta_k}{2}\right)
\end{pmatrix}.
\end{align*}
Note that $0\in c_0(\mathbb{R}^2)$ is a solution of \eqref{application} for all $\lambda\in T^k$, as we require $h$ to satisfy (A1). In what follows we denote by $B\subset T^k$ the set of all bifurcation points of \eqref{application}. The aim of this section is to prove the following theorem.

\begin{theorem}\label{thm:application}
Let $a_n\colon T^k\rightarrow M(2,\mathbb{R})$ and $h_n\colon T^k\times\mathbb{R}^2\rightarrow\mathbb{R}^2$, $n\in\mathbb{Z}$, be sequences of maps as above.
\begin{itemize}
	\item If $k=1$, i.e. $T^1$ is the unit circle, then $B\neq\emptyset$.
	\item If $k\geq 2$, then the covering dimension of $B$ is at least $k-1$ and $B$ is not contractible.
\end{itemize}
\end{theorem}
\begin{proof}
We set
\[f_n:T^k\times\mathbb{R}^2\rightarrow\mathbb{R}^2,\quad f_n(\lambda,u)=a_n(\lambda)u+h_n(\lambda,u),\quad n\in\mathbb{Z},\]
and our first aim is to show that (A1)-(A5) are satisfied, which is clear for (A1)-(A2). Now
\[(Df_n)(\lambda,0)=a_n(\lambda)+(Dh_n)(\lambda,0),\]
and we see that by (B1)
\[\lim_{n\rightarrow-\infty}(Df_n)(\lambda,0)=a(1,\ldots,1),\qquad\lim_{n\rightarrow+\infty}(Df_n)(\lambda,0)=a(\lambda)\]
uniformly in $\lambda$. As the eigenvalues of $a(\lambda)$ are $\frac{1}{2}$ and $2$ for all $\lambda\in T^k$, it follows that (A3) and (A4) hold.\\
It remains to show (A5), for which we consider the parameter value $\lambda_0$ from (B2). We introduce the operator
\[\widetilde{L}:c_0(\mathbb{R}^2)\rightarrow c_0(\mathbb{R}^2),\quad (\widetilde{L}x)_n=x_{n+1}-a_n(\lambda_0)x_n,\]
and we note that it is Fredholm of index $0$ according to Theorem \ref{thm-lin}. Hence $\widetilde{L}$ is invertible if and only if it has a trivial kernel. If now $x\in\ker\widetilde{L}$, then
\begin{align*}
x_n=\begin{cases}
a(\lambda_0)^nx_0,&n\geq 0\\
a(1,\ldots,1)^nx_0, &n<0
\end{cases},
\end{align*}
which yields a non-trivial element in $c_0(\mathbb{R}^2)$ if and only if $x_0\in E^s(a(\lambda_0))\cap E^u(a(1,\ldots,1))$. Clearly, $E^u(a(1,\ldots,1))=\{0\}\oplus\mathbb{R}$, and moreover by a straightforward computation we obtain
\[E^s(a(\lambda_0))=\{u\in \R^2\mid u=t(\cos((\Theta_1+\cdots +\Theta_k)/2),\sin((\Theta_1+\cdots +\Theta_k)/2))\}.\]
Hence $\widetilde{L}$ is invertible as $\lambda_0\notin S$. By the Neumann series, every sufficiently small perturbation of an invertible operator is still invertible. Hence there is a constant $C(\lambda_0)$ such that
\[x_{n+1}=a_n(\lambda_0)x_n+(Dh_n)(\lambda_0,0)x_n,\quad n\in\mathbb{Z}\]
has only the trivial solution in $c_0(\mathbb{R}^2)$ if $\sup_{n\in\mathbb{Z}}\|(Dh_n)(\lambda_0,0)\|<C(\lambda_0)$.\\
Consequently, the family of discrete dynamical systems \eqref{application} satisfies all assumptions (A1)--(A5). Let us now consider the stable bundles $E^s(+\infty)$ and $E^s(-\infty)$. Clearly,
\begin{equation*}
E^s(-\infty)=\{(\lambda,u)\in T^k\times\R^2 \mid \, u=(t,0),\, t\in\R\},
\end{equation*}
and moreover it is readily seen that
\begin{equation*}
E^s(+\infty)=\{(\lambda,u)\in T^k\times\R^2 \mid u=t(\cos((\Theta_1+\cdots +\Theta_k)/2),\sin((\Theta_1+\cdots +\Theta_k)/2)),\, t\in\R\}.
\end{equation*}
As $E^s(-\infty)$ is a trivial bundle, we get that $w_1(E^s(-\infty))=0\in H^1(T^k;\mathbb{Z}_2)$. Let us now consider $E^s(+\infty)$. We have a canonical embedding $\iota:T^1\hookrightarrow T^k$ given by $\lambda_1\mapsto(\lambda_1,1,\ldots,1)$. We pullback $E^s(+\infty)$ to $T^1$ by $\iota$ and obtain
\begin{equation*}
\iota^\ast(E^s(+\infty))=\left\{(\lambda_1,u)\in T^1\times\mathbb{R}^2\mid u=t\left(\cos\left(\Theta_1/2\right),\sin\left(\Theta_1/2\right)\right)\right\},
\end{equation*}
where as before $\lambda_1=\exp(i\Theta_1)$ for $\Theta_1\in(-\pi,\pi]$. Now $\iota^\ast(E^s(+\infty))$ is the Möbius bundle over $S^1$, which is non-orientable. As a vector bundle is orientable if and only if its first Stiefel-Whitney class vanishes (cf. \cite[Thm. II.1.2]{Lawson}), it follows that $w_1(\iota^\ast(E^s(+\infty)))\neq 0\in H^1(S^1;\mathbb{Z}_2)$. We obtain from the naturality of Stiefel-Whitney classes that
\[\iota^\ast(w_1(E^s(+\infty)))=w_1(\iota^\ast(E^s(+\infty)))\neq 0\]
which implies that $w_1(E^s(+\infty))\neq 0\in H^1(T^k;\mathbb{Z}_2)$. Now the assertion of Theorem \ref{thm:application} follows from Corollary \ref{cor-Jacobo} for $k=1$ and from Theorem \ref{thm:nonlinII} for $k\geq 2$.
\end{proof}
\noindent
Let us point out that, using the convention that the covering dimension of the empty set is $-1$, we could write in Theorem \ref{thm:application} that $\dim B\geq k-1$ for all $T^k$. However, the assertion that $B$ is not contractible is in general wrong for $k=1$. Indeed, if we assume in the example above that $h_n\equiv0$ for all $n\in\mathbb{Z}$, then the family \eqref{application} is
\[x_{n+1}=a_n(\lambda)x_n,\quad n\in\mathbb{Z},\]
and clearly the bifurcation points are those $\lambda\in T^1$ for which these linear equations have a non-trivial solution. Arguing as in the proof of Theorem \ref{thm:application}, we see that the space of solutions is isomorphic to $E^s(a(\lambda))\cap E^u(a(1))$. As $E^u(a(1))=\{0\}\oplus\mathbb{R}$ and
\[E^s(a(\lambda))=\left\{t\left(\cos\left(\Theta/2\right),\sin\left(\Theta/2\right)\right)\mid t\in\mathbb{R}\right\},\]
we see that this intersection is non-trivial only if $\lambda=-1$, which shows that $B=\{-1\}\subset T^1$.

\section{An Example concerning Assumption (A5)}
The aim of this section is to show that Theorem \ref{thm:nonlin} is wrong if we do not require Assumption (A5).\\
Let us first introduce some notation and recapitulate basic facts from functional analysis. The dual space $c_0(\mathbb{R}^N)'$ of $c_0(\mathbb{R}^N)$ is canonically isomorphic to $\ell^1(\mathbb{R}^N)$ and the identification is given by assigning to a linear functional $f:c_0(\mathbb{R}^N)\rightarrow\mathbb{R}$ the unique $u\in\ell^1(\mathbb{R}^N)$ such that

\[f(v)=\sum_{n\in\mathbb{Z}}{\langle v_n,u_n\rangle}.\]
We now assume that $a_n:\Lambda\rightarrow M(N,\mathbb{R})$, $n\in\mathbb{Z}$, is a sequence of continuous maps such that

\begin{equation*}
\sup_{n\in\Z}\|a_n(\lambda)\|<\infty \text{ for all }\lambda\in\Lambda.
\end{equation*}
We obtain a family of bounded operators $L\colon \Lambda\times c_0(\mathbb{R}^N)\rightarrow c_0(\mathbb{R}^N)$ given by

\begin{align}\label{Linier}
(L_\lambda x)_n=x_{n+1}-a_n(\lambda)x_n,
\end{align}
and the dual operator $L'\colon \Lambda\times \ell^1(\mathbb{R}^N)\rightarrow \ell^1(\mathbb{R}^N)$ to $L$ is given by

\begin{align}
(L_\lambda' x)_n=x_{n}-a_{n+1}^t(\lambda)x_{n+1},\quad n\in\mathbb{Z},
\end{align}
where we denote by $a^t$ the transpose of a matrix $a\in M(N,\mathbb{R})$. It is well known that $L_\lambda$ is Fredholm if and only if $L'_\lambda$ is Fredholm, and

\begin{align}\label{ortogonal}
\begin{split}
\ker (L'_{\lambda})&= \im(L_{\lambda})^{\perp}:=\{f\in c_0(\mathbb{R}^N)':\,f(z)=0\,\text{for all}\, z\in\im(L_\lambda) \}\\
&=\left\{z\in\ell^1(\mathbb{R}^N):\,\sum_{n\in\mathbb{Z}}{\langle z_n,v_n\rangle=0}\,\,\text{for all}\,\, v\in\im( L_\lambda)\right\}.
\end{split}
\end{align}
Let us now take once again as parameter space $\Lambda=T^k$ as in Section $6$, and let us consider the family of discrete dynamical systems
\begin{align}\label{DDS-with-h}
x_{n+1}=a_n(\lambda)x_n+h_n(\lambda,x_n)\quad \text{ for } \lambda=(\lambda_1,\ldots,\lambda_k)\in T^k,\quad n\in\Z,
\end{align}
where
\begin{align*}
h_n(\lambda,x)&=(0,0,0,\|x\|^2)\quad\,\text{ for }\, x\in\R^4,\, n\in\Z,
\end{align*}

\begin{align*}
a_n(\lambda)=\begin{cases}
a(\lambda),\quad &n\geq 0\\
a_-,\quad &n<0
\end{cases}
\end{align*}
and
\begin{align*}
a(\lambda)&=\begin{pmatrix}
\frac{1}{2}+\frac{3}{2}\sin^2\left(\frac{\Theta_1+\cdots+\Theta_k}{2}\right)&-\frac{3}{4}\sin(\Theta_1+\cdots+\Theta_k) & 0 & 0\\
-\frac{3}{4}\sin(\Theta_1+\cdots+\Theta_k)&\frac{1}{2}+\frac{3}{2}\cos^2\left(\frac{\Theta_1+\cdots+\Theta_k}{2}\right) & 0 & 0\\
0 & 0 & \frac{1}{2} & 0\\
0 & 0 & 0 & 2
\end{pmatrix},
\\
a_-&=\begin{pmatrix}
\frac{1}{2}&0 & 0 & 0\\
0 & 2& 0 & 0\\
0 & 0 & 2 & 0\\
0 & 0 & 0 & \frac{1}{2}
\end{pmatrix}.
\end{align*}
Note that for any $\lambda\in T^k$ the linear system $x_{n+1}=a_n(\lambda)x_n$ admits the non-trivial solution
\begin{equation*}
x=(x_n)=((0,0,1/2^{|n|},0))\in c_0(\R^4).
\end{equation*}
Our aim is now to show that for any $\lambda\in T^k$ the discrete dynamical system $x_{n+1}=a_n(\lambda)x_n+h_n(\lambda,x_n)$ does not have a nontrivial
solution, which implies that there are no bifurcation points.
\newline\indent For this purpose, let us consider the operators $L_{\lambda}\colon c_0(\R^4)\to c_0(\R^4)$ given by

\begin{align*}
(L_{\lambda}x)_n&=x_{n+1}-\;a_n(\lambda)x_n \quad \text{ for } x\in c_0(\R^4),\;\lambda\in T^k,\\
\end{align*}
and let us note that the corresponding dual operator $L_{\lambda}'\colon \ell^1(\R^4)\to \ell^1(\R^4)$ is given by

\begin{align*}
(L_{\lambda}'y)_n&=y_{n}-a_{n+1}^t(\lambda)y_{n+1} \quad \text{ for } y\in \ell^1(\R^4),\;\lambda\in T^k.
\end{align*}
It is readily seen that

\begin{equation}
y=(y_n)=((0,0,0,1/2^{|n|}))\in \ell^1(\R^4)
\end{equation}
is in the kernel of $L'_\lambda$ for all $\lambda\in T^k$. Let now $u=(u_n)$ be a solution of
\eqref{DDS-with-h} and let us recall that our aim is to show that $u=0$. As $v=(h_n(\lambda,u_n))\in \im L_{\lambda}$ and $y\in \ker L'_{\lambda}$, we see from \eqref{ortogonal} that

\begin{align*}
\langle y,v\rangle=\sum_{n\in \Z}\langle y_n,v_n\rangle=0.
\end{align*}
Consequently,
\begin{align*}
0=\sum_{n\in \Z}\langle y_n,v_n\rangle=\sum_{n\in \Z}\langle y_n,h_n(\lambda,u_n)\rangle=\sum_{n\in \Z}\frac{1}{2^{|n|}}\cdot \|u_n\|^2,
\end{align*}
and so $\|u_n\|=0$ for all $n\in\Z$. Hence we have shown that \eqref{DDS-with-h} has no bifurcation point.\\
However, it is readily seen that
\begin{equation*}
E^s(-\infty)=\{(\lambda,u)\in T^k\times\R^4\mid \, u=(t,0,0,s),\, t,s\in\R\}
\end{equation*}
and
\begin{equation*}
E^s(+\infty)=\{(\lambda,u)\in T^k\times\R^4\mid u=(t\cos((\Theta_1+\cdots +\Theta_k)/2),t\sin((\Theta_1+\cdots +\Theta_k)/2),s,0), t,s\in\mathbb{R}\}.
\end{equation*}
It can be shown as in Section $6$ that
\begin{equation*}
w(E^s(+\infty))\neq w(E^s(-\infty))\in H^\ast(T^k;\Z_2).
\end{equation*}
Hence we have found an example of a discrete dynamical system satisfying Assumptions (A1)--(A4) except of (A5) which does not satisfy the conclusion of Theorem \ref{thm:nonlin}.

%%%%%%%%%%%%%%%%%%%%%%%%%%%%%%%%%%%%%%%%%%%%%%%%%%%%%%%%%%%%%%%%%%%%%%%%%%%%%%%%%%%%%%%%%%%%%%%%%%%%%%%%%%%%%%%%%%%%%%%%%%%%%%%%%%%%%%%%%%%%%%%%%%%%%%%%%%%%%%%%%%%%%%%%%%%%%%%%%%%%%%%%%%%%%%%%%%%%%%%%%%%%%%%%%%%%%%%%%%%%%%%%%%%%%%%%%%%%%%%%%%%%%%%%%%%%%%%%%%%%%%%%%%%%%%%%%%%%%%%%%%%%%%%%%%%%%%%%%%%%%%%%%%%%%%%%%%%%%%%%%%%%%%%%%%%%%%%%%%%%%%%%%%%%%%%%%%%%%%%%%%

\section{Concluding remarks}
There are several interesting points of further study which are not covered in this paper and of which we just want to mention the following ones:

\begin{itemize}
\item Theorem \ref{thm-lin} actually holds true under weaker assumptions (and hence also the Theorems \ref{thm:nonlin} and \ref{thm:nonlinII}). Namely, it suffices to assume that the linearised systems

\[x_{n+1}=a_n(\lambda)x_n\]
admit an exponential dichotomy on $\Z^{\pm}$ $($with corresponding projectors $P^{\pm}\colon \Lambda\times\Z^{\pm}\to M(N,\R))$, where $\Z^+:=\Z\cap [0,\infty)$ and $\Z^-:=\Z\cap (-\infty,0]$. Then the index bundle of $L$ given by \eqref{Linier} is of the form:	
\begin{equation*}
\ind(L)=[\im P^+_0]-[\im P^-_0]\in KO(\Lambda),
\end{equation*}
where $\im P^{\pm}_{0}:=\{(\lambda,x)\in \Lambda\times \R^N\mid x\in \im P^{\pm}(\lambda,0)\}$ denotes the vector bundles induced by the projectors $P^{\pm}$.  The proof involves some additional techniques which are not contained and discussed here and will be treated in a forthcoming paper.

The concept of an exponential dichotomy (ED for short) introduced by Perron \cite{Perron} plays a central role in the stability
theory of differential equations, discrete dynamical systems, delay evolution equations and many others fields of mathematics.
The concept was taken forward by Coppel \cite{Coppel}, Palmer \cite{Pa84} and others. In particular, a significant contribution in this direction, in infinite dimensional spaces, was made by Henry in \cite{Henry}, where he carried over this concept from the topic of differential equations to the case of discrete dynamical systems.
Note that an exponential dichotomy extends the idea of hyperbolicity for autonomous discrete dynamical systems
to explicitly non autonomous discrete dynamical systems. More precisely, an ED is a hyperbolic splitting of the extended
state space for linear non autonomous difference equations into two
vector bundles. The first one, called stable vector bundle, consists of all solutions decaying exponentially in forward time, while the complementary unstable
vector  bundle  consists  of  all  solutions  which  exist  and  decay  in  backward  time. Moreover, an ED allows to provide
a necessary condition for bifurcations of entire solutions (see for example \cite{Christian10}).
\item Our methods can easily be adapted in order to study bifurcation of homoclinics on manifolds. Following \cite{Ab-Ma2}, to each finite dimensional manifold $M$ and  discrete dynamical system $f\colon \Z\times M\to M$ having  $x\in M$  as a stationary trajectory, we can associate the Banach manifold $c_x(M)$ which is a natural place for the study of  trajectories of the dynamical system $f$ homoclinic to $x$.
\item In this paper we have not studied global bifurcation of homoclinic trajectories. It should turn out that the methods developed in this article can be applied to the study of the existence of
connected branches  of solutions  and their  behavior, and the  existence of  large  homoclinic trajectories  using bifurcation from infinity. Some partial results were obtained in \cite{JacoboRobertI}. However, this subject requires more thorough investigations.
\item Another interesting topic which will be considered in a forthcoming paper concerns bifurcation of homoclinic trajectories
of non-autonomous Hamiltonian vector fields parametrised by a compact and connected space. This topic has been considered, e.g. in \cite{Jacobo, Secchi}, but it also shall be treated by our methods.
\end{itemize}

\section*{Acknowledgements}
The first author is supported by the NCN Grant 2013/09/B/ST1/01963.

\newpage

\vspace{1cm}
Robert Skiba\\
Faculty of Mathematics and Computer Science\\
Nicolaus Copernicus University\\
Chopina 12/18\\
87-100 Torun\\
Poland\\
E-mail: robo@mat.umk.pl

\vspace{1cm}
Nils Waterstraat\\
School of Mathematics,\\
Statistics \& Actuarial Science\\
University of Kent\\
Canterbury\\
Kent CT2 7NF\\
UNITED KINGDOM\\
E-mail: n.waterstraat@kent.ac.uk

\end{document}